\documentclass[preprint,letterpaper,10pt]{elsarticle}

\usepackage{amsmath}
\usepackage{amssymb}
\usepackage{amsthm}

\usepackage{empheq}
\usepackage{xcolor}

\usepackage[margin=1.2in]{geometry}

\usepackage[colorlinks=true,linkcolor=blue,filecolor=blue,urlcolor=blue,citecolor=blue,pdftex,plainpages=false]{hyperref}

\usepackage{verbatim}

\newdefinition{defi}{Definition}
\newtheorem{lem}{Lemma}
\newtheorem{thm}{Theorem}
\newdefinition{rmk}{Remark}

\usepackage{lineno}

\begin{document}

\begin{frontmatter}

\title{2N-storage Runge-Kutta methods:\\
Order conditions, general properties and some analytic solutions}

\author{Alexei Bazavov}
\ead{bazavov@msu.edu}
\address{Department of Computational Mathematics, Science and Engineering and\\
Department of Physics and Astronomy,\\
Michigan State University, East Lansing, MI 48824, USA}

\begin{abstract}
Low-storage Runge-Kutta schemes of Williamson's type, so-called 2N-storage schemes, are examined. Explicit 2N-storage constraints are derived for the first time and used to establish new relations between the entries of the Butcher tableau. An error in the Williamson's formula for converting coefficients between the standard and 2N-storage formats in the special case is pointed out and corrected. The new relations are used to derive a closed-form solution for four- and five-stage 2N-storage methods with the third order of global accuracy. Several new four- and five-stage schemes with rational coefficients are presented and numerically examined for illustration.
\end{abstract}


\end{frontmatter}


\section{Introduction}
\label{sec_intro}
 Runge-Kutta methods are well-known and widely used methods for numerically solving initial value problems. A variety of literature is dedicated to them, and one can find a standard textbook treatment in Refs.~\cite{ButcherBook,HairerBook1}. It has been found that explicit Runge-Kutta methods may allow for low-storage implementation when the number of memory locations for storing the intermediate stages can be reduced. Imposing additional constraints leads to so-called low-storage Runge-Kutta methods that possess structures different from the standard ones. The main focus of this work is not the memory considerations, but those structural properties of the low-storage methods, more narrowly, 2N-storage methods of Williamson's type, defined in Sec.~\ref{sec_over_2n}. Since their introduction by Williamson in 1980~\cite{WILLIAMSON198048}, 2N-storage Runge-Kutta methods were developed further by a number of authors~\cite{CK1994,STANESCU1998674,BERLAND20061459,ALLAMPALLI20093837,BERNARDINI20094182,KETCHESON20101763,TOULORGE20122067,NIEGEMANN2012364,Yan2017}. An $s$-stage Runge-Kutta method of the global order of accuracy $p$ is denoted $(s,p)$ in this paper.
 
 Numerical evidence, recently examined in Ref.~\cite{Bazavov2021} strongly suggests that the 2N-storage Runge-Kutta methods of Williamson's type possess quite a remarkable property: they are automatically Lie group, or structure preserving integrators. No extra order conditions or commutators are needed to turn a classical 2N-storage method into a Lie group integrator. This result was proved for a (3,3) 2N-storage method and remains a conjecture for methods with more stages and of higher order. Most of the 2N-storage methods present in the literature to date were examined in Ref.~\cite{Bazavov2021} and the results support the conjecture.
 That structure-preserving feature of the 2N-storage methods motivated deeper investigation of their properties, which resulted in the present work. The paper is organized as follows. General definitions and established results are given in Sec.~\ref{sec_over_2n}. In Sec.~\ref{sec_equivalent} an intermediate, so-called, $\alpha$-form of the 2N-storage method is introduced. In Sec.~\ref{sec_order} the additional constraints on the Runge-Kutta coefficients that make the scheme a 2N-storage method are derived in an explicit form for the first time. Together with the standard order conditions these constraints form a set of nonlinear equations, called 2N-storage order conditions, that can be solved in the original Runge-Kutta parameters, rather than in the low-storage parameters as has been done in previous work. In Sec.~\ref{sec_relations} relations between the standard and 2N-storage forms are examined from the perspective of the explicit 2N-storage order conditions derived in Sec.~\ref{sec_order}, and an error in the special case of Williamson's formulas~\cite{WILLIAMSON198048} is pointed out and corrected. A Matlab script to help the reader quickly see the issue is included in \ref{sec_app_matlab}. In Sec.~\ref{sec_abc} the 2N-storage constraints are explicitly solved, in the sense that all $a_{ij}$ parameters of a 2N-storage Runge-Kutta method, independently of its order of accuracy, can be expressed through the nodes $c_i$ and weights $b_i$. In Secs.~\ref{sec_lsrk43} and \ref{sec_lsrk53} the new results are applied to derive a general solution for the (4,3) and (5,3) methods in the following sense. A single second-order equation for the $b_2$ parameter is presented whose coefficients are functions of the free parameters. Once the latter are chosen and the equation is solved for $b_2$, the whole Butcher tableau is reconstructed from simple relations, which are also presented. Special cases for the (4,3) method are presented for illustrative purposes in \ref{sec_app_spec43}. As the equation for $b_2$ is of such a low order, schemes with rational coefficients are found easily (over two million were found relatively quickly with a simple grid search). Several (4,3) and (5,3) schemes with rational coefficients are presented and examined numerically in Sec.~\ref{sec_num} for illustration. The conclusions are given in Sec.~\ref{sec_concl}.

\section{Overview of the standard and 2N-storage Runge-Kutta methods}
\label{sec_over_2n}
The starting point is an initial value problem consisting of a first order differential equation for a function $y(t)$:
\begin{equation}
\label{eq_dydt}
\frac{dy}{dt}=f(t,y)
\end{equation}
and the specified initial condition $y(t_0)$.
An $s$-stage explicit Runge-Kutta method which is referred to as \textit{standard} or \textit{classical} in the following, for numerically integrating Eq.~(\ref{eq_dydt}) from time $t$ to $t+h$ can be written in the form~\cite{ButcherBook,HairerBook1}:
\begin{eqnarray}
y_i&=&y_t + h\sum_{j=1}^{s}a_{ij}k_j,\label{eq_yi}\\
k_i&=&f(t+hc_i,y_i),\\
i&=&1,\dots,s,\\
y_{t+h}&=&y_t+h\sum_{i=1}^{s}b_ik_i\label{eq_yth},
\end{eqnarray}
where $y_t$ is a numerical approximation of $y(t)$ at the beginning of the step and
$y_{t+h}$ is the one at the end. The step size is denoted $h$.
Only explicit methods are considered here, thus $a_{ij}=0$ for $i\leqslant j$. The coefficients $a_{ij}$, $b_i$ and $c_i$ are often arranged in the form of a Butcher tableau~\cite{ButcherBook}, as shown in~\ref{sec_app_btab}. The notation for a sum (product) implies that when the upper bound on the index is smaller than the lower bound, the sum (product) is set to 0 (1), \textit{e.g.}, $\sum_{j=1}^0\dots=0$ and $\prod_{j=1}^0\dots=1$. The self-consistency conditions relate $a_{ij}$ and $c_i$ coefficients as
\begin{equation}
\label{eq_c_from_a}
c_i=\sum_{j=1}^{i-1}a_{ij},
\end{equation}
and as a result for an explicit method $c_1=0$. For the later convenience, $y_0\equiv y_t$ and $y_{s+1}\equiv y_{t+h}$ are also introduced.

\begin{defi}
In the context of this paper the standard form of a Runge-Kutta method introduced in Eqs.~(\ref{eq_yi})--(\ref{eq_yth}) is called the $a$-form.
\end{defi}

A standard method requires storing right-hand-side evaluation $k_i$ from all stages to be applied at the final stage, Eq.~(\ref{eq_yth}).

Williamson~\cite{WILLIAMSON198048} introduced a class of Runge-Kutta methods that require only two storage locations per step. For a system of $N$ degrees of freedom $2N$ storage locations are required and such methods are commonly referred to as 2N-storage methods, while a wider class of methods that use 2, 3 or more storage locations is referred to as \textit{low-storage methods}. Here only 2N-storage methods of Williamson's type are discussed. They can be cast into the following form:
\begin{eqnarray}
\Delta y_i &=& A_{i}\Delta y_{i-1}+hf(t+c_{i}h,y_{i-1}),\label{eq_2N_W_dyi}\\
y_i &=& y_{i-1} + B_{i}\Delta y_i,\\
i &=&1,\dots,s.\label{eq_2N_W_i}
\end{eqnarray}
\begin{defi}
The representation of the method in Eqs.~(\ref{eq_2N_W_dyi})--(\ref{eq_2N_W_i}) is called the $A$-form.
\end{defi}

It is understood that $\Delta y_i$, $y_i$ overwrite $\Delta y_{i-1}$, $y_{i-1}$ at each stage of the method. While the coefficients of a standard Runge-Kutta method satisfy the corresponding order conditions, the low-storage coefficients $A_i$, $B_i$ are subject to additional constraints and not every standard Runge-Kutta method can be cast into 2N-storage form. The coefficients $a_{ij}$, $b_i$, $c_i$ of the $a$-form of the method are related to the coefficients $A_i$, $B_i$ of the $A$-form as:
\begin{eqnarray}
B_i &=& \left\{
\begin{array}{ll}
a_{i+1,i}, & i=1,\dots,s-1,\\
b_s, & i=s,
\end{array}
\right.\label{eq_W_ab_to_B}\\
A_i &=&
\begin{array}{ll}
\displaystyle
\frac{b_{i-1}-B_{i-1}}{b_{i}}, & i\neq1,\,\,b_i\neq0,\label{eq_W_ab_to_A}\\
\end{array}\\
A_i &=&
\begin{array}{ll}
\displaystyle
\frac{a_{i+1,i-1}-c_{i}}{B_{i}}, & i\neq1,\,\,b_i=0.\label{eq_W_ab0_to_A}\\
\end{array}
\end{eqnarray}
Equations~(\ref{eq_W_ab_to_B})--(\ref{eq_W_ab0_to_A}) first appeared in Ref.~\cite{WILLIAMSON198048}. However, the special case, Eq.~(\ref{eq_W_ab0_to_A}) is only valid for $b_2=0$ and is incorrect for $b_i=0$, $i>2$ as will be shown in Sec.~\ref{sec_Wcorr_b0}. The correct form is presented in Sec.~\ref{sec_Wcorr_b0}.
The inverse transformation is
\begin{eqnarray}
a_{ij} &=& \left\{
\begin{array}{ll}
A_{j+1}a_{i,j+1}+B_j, & j<i-1,\\
B_j, & j=i-1,\\
0, & \mbox{otherwise},
\end{array}
\right.\label{eq_Ai}\\
b_i &=&\left\{
\begin{array}{ll}
A_{i+1}b_{i+1}+B_i, & i<s,\\
B_i, & i=s.
\end{array}
\right.\label{eq_Bi}
\end{eqnarray}

A typical construction of a 2N-storage method proceeds in the following way: the $a$-form coefficients are expressed through the $A$-form coefficients, Eqs.~(\ref{eq_Ai}), (\ref{eq_Bi}). The former are then substituted into the standard order conditions, shown up to and including the order 3 in \ref{sec_app_oc}, which are solved for $A_i$, $B_i$. The additional constraints that guarantee that the method can be written in 2N-storage form are always implicit -- they are encoded in Eqs.~(\ref{eq_Ai}), (\ref{eq_Bi}). In Sec.~\ref{sec_order} these constraints are derived explicitly in the $a$-form, \textit{i.e.}, additional relations between the coefficients $a_{ij}$, $b_i$ that they need to satisfy for a 2N-storage method are presented. Together with the standard order conditions they form a larger set of equations for $a_{ij}$, $b_i$ that is called the \textit{order conditions for 2N-storage methods} here. As is shown in Sec.~\ref{sec_order}, the additional 2N-storage constraints are always quadratic in $a_{ij}$, $b_i$ independently of the order of the method. Thus, this new explicit form of the order conditions for 2N-storage methods may be easier to solve. Moreover, apart from special cases (\textit{e.g.}, $c_2=c_3$, $c_3=c_4$, etc.), in the general case, the additional constraints can be explicitly solved and the coefficients $a_{ij}$ of a 2N-storage method are completely determined by $b_i$ and $c_i$. This allows one to express $A_i$, $B_i$ in terms of $b_i$ and $c_i$ as well.

\section{Equivalent forms of 2N-storage methods}
\label{sec_equivalent}
Before deriving the additional 2N-storage order conditions the following two auxilliary steps are needed: a) resolving the stage indexing inconsistency between the standard, Eqs.~(\ref{eq_yi})--(\ref{eq_yth}), and 2N-storage, Eqs.~(\ref{eq_2N_W_dyi})--(\ref{eq_2N_W_i}), representations, and b) introducing an intermediate form of the method that naturally connects the $a$-form to the $A$-form.

Consider the $a$-form, Eqs.~(\ref{eq_yi})--(\ref{eq_yth}). The reader is reminded that only explicit methods, $a_{1,i}=0$ for any $i$, $c_1=0$ are considered here. Thus, at $i=1$ simply $y_1=y_0$. Then $k_1=f(t+c_1h,y_1)=f(t,y_0)$ and $y_2=y_0+ha_{21}k_1$. The first non-trivial stage for an explicit method in the $a$-form is $i=2$.

In the $A$-form of a 2N-storage method, however, $\Delta y_1=A_1\Delta y_0+hf(t,y_0)$ (since $\Delta y_0$ is not available, for a self-starting explicit method always $A_1=0$) and $y_1=y_0+B_1\Delta y_1=y_0 + hB_1f(t,y_0)=y_0+hB_1k_1$. Thus, $y_1$ of the $A$-form (as it is usually defined in the literature) is equivalent to $y_2$ in the $a$-form. To derive the order conditions, the stages need to be equated. As the $a$-form of a Runge-Kutta method, Eqs.~(\ref{eq_yi})--(\ref{eq_yth}) is now standard in the majority of the literature, it is taken as a baseline and the indexing is shifted by 1 in the $A$-form. Before that, the following slight modification of the $a$-form is introduced that will become convenient later. While in the $a$-form the last stage with the $b_i$ coefficients is distinguished from the previous ones, in the 2N-storage $A$-form all stages are equivalent, therefore let
\begin{equation}
a_{s+1,i}\equiv b_i
\end{equation}
and
\begin{equation}
c_{s+1}\equiv 1.
\end{equation}
Then a modified $a$-form where all stages are treated the same is
\begin{eqnarray}
y_i &=& y_0 + h\sum_{j=1}^{s} a_{ij}k_j,\\
k_i &=& f(t+c_ih, y_i),\\
i &=& 1,\dots,s+1.
\end{eqnarray}
Then $y_{s+1}$ is the sought $y_{t+h}$ and $k_{s+1}$ is not used since it is simply $k_1$ on the \textit{next} step of the integration.

Adding for convenience $A_0=B_0=c_0=0$ to the coefficient set, the shifted 2N-storage $A$-form can be represented as
\begin{eqnarray}
\Delta y_i &=& A_{i-1}\Delta y_{i-1}+hf(t+c_{i-1}h,y_{i-1}),\label{eq_2N_dyi}\\
y_i &=& y_{i-1} + B_{i-1}\Delta y_i,\\
i &=&1,\dots,s+1.\label{eq_2N_i}
\end{eqnarray}
Then $y_1=y_0$ ($\Delta y_1$ does not matter since $B_0=0$), $\Delta y_2=A_1\Delta y_1+hf(t+c_1h,y_0)=hk_1$, $y_2=y_1+B_1\Delta y_2=y_0+hB_1k_1$.

As $B_1=a_{21}$ (from Eq.~(\ref{eq_W_ab_to_B})), there is now a proper match in the indexing of the $a$-form and $A$-form. This is merely a shift in indexing that does not affect the relations between the coefficients, Eqs.~(\ref{eq_W_ab_to_B})--(\ref{eq_Bi}).

Next, there is still an important difference between the $a$-form and $A$-form. In the former the approximation $y_i$ at stage $i$ is evaluated from $y_0$ (the beginning of the step), while in the latter $y_i$ is evaluated from $y_{i-1}$ (previous stage).

An equivalent form of a standard Runge-Kutta method with \textit{propagation from the previous stage} is also possible:
\begin{eqnarray}
y_i &=& y_{i-1} + h\sum_{j=1}^{i-1}\alpha_{ij}k_j,\label{eq_alpha_yi}\\
k_i &=& f(t+c_ih,y_i),\label{eq_alpha_ki}\\
i&=&1,\dots,s+1,\label{eq_alpha_i}
\end{eqnarray}
where again the method is explicit, $\alpha_{ij}=0$ for $i\leqslant j$ and $\alpha_{s+1,i}\equiv \beta_i$.

\begin{defi}
Eqs.~(\ref{eq_alpha_yi})--(\ref{eq_alpha_i}) are referred to as the $\alpha$-form of the method.
\end{defi}
\begin{rmk}
Eq.~(\ref{eq_alpha_yi}) is similar to the Shu-Osher form~\cite{SHU1988439}. The meaning of $\alpha_{ij}$ is different (they are typically denoted $\beta_{ij}$ in that form). While a 2N-storage method can be written in the Shu-Osher representation, it is not used here, therefore the notation should not cause any confusion.
\end{rmk}

After recursive substitution of $y_{i-1}$ into $y_i$ one finds:
\begin{equation}
\label{eq_a_from_alpha}
a_{ij} = \sum_{k=j+1}^{i}\alpha_{kj}
\end{equation}
and for the last stage $i=s+1$
\begin{equation}
\label{eq_b_from_beta}
b_i=\beta_i+\sum_{k=i+1}^s\alpha_{ki}.
\end{equation}
The inverse transformation is
\begin{eqnarray}
\alpha_{ij}&=&a_{ij}-a_{i-1,j},\label{eq_alpha_a}\\
\beta_i&=&b_i - a_{s,i}.\label{eq_beta_b}
\end{eqnarray}
Equation~(\ref{eq_alpha_a}) implies that for the elements on the diagonal of the Butcher tableau $a_{i,i-1}=\alpha_{i,i-1}$ and Eq.~(\ref{eq_beta_b}) that $\beta_s=b_s$. The mapping between the $a$-form and $\alpha$-form is bijective. There is the same number of $\alpha_{ij}$ as $a_{ij}$ and $\beta_i$ as $b_i$. While any explicit Runge-Kutta method can be represented in the $a$-form or equivalently in the $\alpha$-form, only the 2N-storage subset of methods can be consistently represented in the $A$-form.

As the $\alpha$-form corresponds more closely to the $A$-form in that it updates from the previous stage rather than from the beginning of the step, the relations between the $\alpha_{ij}$, $\beta_i$ and $A_i$, $B_i$ coefficients can now be derived. Consider the first few stages of the $A$-form (substituting $\Delta y_i$ explicitly):
\begin{eqnarray}
\Delta y_2 &=& hk_1,\\
\Delta y_3 &=& hA_2k_1+hk_2,\\
\Delta y_4 &=& hA_3A_2k_1 + hA_3k_2 + h k_3.
\end{eqnarray}
This pattern can be generalized in the following
\begin{lem}
\label{lem1}
The general form of $\Delta y_i$ in the $A$-form, Eqs.~(\ref{eq_2N_dyi})--(\ref{eq_2N_i}), is
\begin{equation}
\label{eq_delta_ym}
\Delta y_i = h\sum_{m=1}^{i-1}\left(\prod_{l=m+1}^{i-1}A_l\right)k_m.
\end{equation}
\end{lem}
\begin{proof}
By induction. For the first non-trivial stage $i=2$ from Eq.~(\ref{eq_delta_ym})
\begin{equation}
\Delta y_2=h\sum_{m=1}^{2-1}\left(\prod_{l=m+1}^{2-1}A_l\right)k_m=hk_1
\end{equation}
which matches Eq.~(\ref{eq_2N_dyi}).
(Recall the convention on the products and sums, \textit{i.e.}, $\prod_{l=2}^{1}\dots\equiv1$.) 
It is now assumed that the form (\ref{eq_delta_ym}) holds for $i=n$. Using Eq.~(\ref{eq_2N_dyi}) with $i=n+1$ and the definition (\ref{eq_alpha_ki})
\begin{equation}
\Delta y_{n+1}=A_n\Delta y_n + hk_n.
\end{equation}
After substituting $\Delta y_n$ and introducing $1=\prod_{l=n+1}^n A_l$ in the second term the result is
\begin{eqnarray}
\Delta y_{n+1}&=&h\sum_{m=1}^{n-1}A_n\prod_{l=m+1}^{n-1}A_lk_m+h\prod_{l=n+1}^n A_lk_n\\
&=&h\sum_{m=1}^{n-1}\prod_{l=m+1}^{n}A_lk_m+h\sum_{m=n}^{n}\prod_{l=m+1}^nA_lk_m\\
&=&h\sum_{m=1}^{(n+1)-1}\left(\prod_{l=m+1}^{(n+1)-1}A_l\right)k_m.
\end{eqnarray}
This completes the proof.
\end{proof}

With the help of Lemma~\ref{lem1} the next step is to prove
\begin{lem}
\label{lem2}
The $\alpha$-form coefficients are related to the $A$-form coefficients as
\begin{eqnarray}
\alpha_{ij}&=&B_{i-1}\prod_{k=j+1}^{i-1}A_k,\,\,\,\,\, i=2,\dots,s,\label{eq_alpha_A}\\
\beta_i &=& B_s\prod_{k=i+1}^{s}A_k,\,\,\,\,\,i=1,\dots,s.\label{eq_beta_B}
\end{eqnarray}
\end{lem}
\begin{proof}
At stage $i$ of the $A$-form, using $\Delta y_i$ given in Eq.~(\ref{eq_delta_ym}), Lemma 1:
\begin{equation}
y_i=y_{i-1}+B_{i-1}\Delta y_i=y_{i-1}+h\sum_{j=1}^{i-1}B_{i-1}\prod_{l=i+1}^{i-1}A_lk_i, \,\,\,\,\,i=2,\dots,s+1.
\end{equation}
Comparing with the $\alpha$-form, Eq.~(\ref{eq_alpha_yi}) for $i=2,\dots,s$ gives Eq.~(\ref{eq_alpha_A}) and for $i=s+1$ and using $\beta_i\equiv\alpha_{s+1,i}$ gives Eq.~(\ref{eq_beta_B}).
\end{proof}

Equations~(\ref{eq_alpha_A}) and (\ref{eq_beta_B}) lead to the following important
\begin{rmk}
For a Runge-Kutta method to be a 2N-storage method with $s$ stages, all of the coefficients $B_i$, $i=1,\dots,s$ must be non-zero. The $A_i$ coefficients for $i=2,\dots,s$ may be zero, however, with rare exceptions~\cite{Ruuth2006SSPRK} almost all 2N-storage methods in the literature have $A_i\neq0$, $i=2,\dots,s$ (and $A_1=0$ for explicit methods). The discussion here is restricted to the 2N-storage methods where all $A_i$ (except $A_1$) are non-zero. Eqs.~(\ref{eq_alpha_A}) and (\ref{eq_beta_B}) then imply that all $\alpha_{ij}$, $i=2,\dots,s$, $j=1,\dots,i-1$ and $\beta_i$, $i=1,\dots,s$ \textit{must also be non-zero}.
\end{rmk}

Note that $\alpha_{j+1,j}=B_j$, $j=1,\dots,s-1$ and $\beta_s=B_s$. As discussed after Eq.~(\ref{eq_beta_b}), this also implies $a_{j+1,j}=B_j$, $b_s=B_s$. Thus, the $B_i$ coefficients are the elements on the diagonal of the $a$-form of the Butcher tableau, as is also evident from Eq.~(\ref{eq_W_ab_to_B}).

Converting the $\alpha$-form to $a$-form in Eqs.~(\ref{eq_alpha_A}) and (\ref{eq_beta_B}) produces the relations between $a_{ij}$, $b_i$ and $A_i$, $B_i$ 
(\ref{eq_W_ab_to_B})--(\ref{eq_Bi})
that have been used in the literature. Before revisiting them, the following discussion is presented using the $\alpha$-form as it allows for deriving the order conditions in general form.

\section{2N-storage constraints for Runge-Kutta methods}
\label{sec_order}
Starting with the $\alpha$-form of a Runge-Kutta method and Eqs.~(\ref{eq_alpha_A}) and (\ref{eq_beta_B}) almost immediately lead to the following
\begin{thm}
\label{thm1}
If a Runge-Kutta method is a 2N-storage method with $A_i\neq0$, $i=2,\dots,s$, then its coefficients in the $\alpha$-form satisfy the following set of equations:
\begin{equation}
\label{eq_oc_2N_alpha}
\frac{\alpha_{j+2,j}}{\alpha_{j+2,j+1}}=
\frac{\alpha_{j+3,j}}{\alpha_{j+3,j+1}}=\dots=
\frac{\alpha_{s,j}}{\alpha_{s,j+1}}=
\frac{\beta_{j}}{\beta_{j+1}},\,\,\,\,\,j=1,\dots,s-2.
\end{equation}
\end{thm}
\begin{proof}
The following ratios (with $j+1<i\leqslant s$) can be formed using Lemma~\ref{lem2}:
\begin{equation}
\label{eq_ratio_alpha}
\frac{\alpha_{i,j}}{\alpha_{i,j+1}}=\frac{B_{i-1}\prod_{k=j+1}^{i-1}A_k}{B_{i-1}\prod_{k=j+2}^{i-1}A_k}=A_{j+1}.
\end{equation}
For $i=s+1$, similarly,
\begin{equation}
\label{eq_ratio_beta}
\frac{\beta_j}{\beta_{j+1}}=\frac{B_{s}\prod_{k=j+1}^{s}A_k}{B_{s}\prod_{k=j+2}^{s}A_k}=A_{j+1}.
\end{equation}
The ratio $\alpha_{i,j}/\alpha_{i,j+1}$ is independent of $i$, thus all ratios for allowed values of $i$ are the same for fixed $j$ and are equal to $\beta_j/\beta_{j+1}$. This gives Eq.~(\ref{eq_oc_2N_alpha}).
\end{proof}

The converse is also true, since, if Eq.~(\ref{eq_oc_2N_alpha}) holds, one can use the ratios (\ref{eq_ratio_alpha}) and (\ref{eq_ratio_beta}) to \textit{define} the coefficients $A_i$ (with $A_1=0$). The coefficients $B_i$ are always given by the elements on the diagonal of the Butcher tableau and thus the resulting Runge-Kutta method has a proper 2N-storage $A$-form.

In principle, Eq.~(\ref{eq_oc_2N_alpha}) is all one needs. Together with the standard order conditions, \textit{e.g}, Eqs.~(\ref{eq_oc_RK_b})--(\ref{eq_oc_RK_bac}) in \ref{sec_app_oc}, the relations in Eq.~(\ref{eq_oc_2N_alpha}) form a set of nonlinear equations to be solved for $a_{ij}$, $b_i$ and $c_i$ to construct a 2N-storage method with desired $(s,p)$ properties. However, Eq.~(\ref{eq_oc_2N_alpha}) is not yet convenient for practical use. There are in total $s(s-1)(s-2)/6$ relations, while only $(s-1)(s-2)/2$ of them are independent. And the 2N-storage conditions (\ref{eq_oc_2N_alpha}) are in the $\alpha$-form, while the standard order conditions are in the $a$-form. The next step is to transform Eq.~(\ref{eq_oc_2N_alpha}) into a set of $(s-1)(s-2)/2$ independent relations in the $a$-form.

\begin{thm}
\label{th_oc_alpha}
For a Runge-Kutta method to be a 2N-storage method, its coefficients in the $a$-form must satisfy the conditions that  can be represented in the following two equivalent ways:

(I)
\begin{equation}
\label{eq_oc_1a}
a_{ij}(a_{i-1,j-1}-a_{j,j-1})=(a_{i,j-1}-a_{j,j-1})a_{i-1,j},
\end{equation}
\begin{equation}
\label{eq_oc_1b}
a_{i,i-1}(b_{i-2}-a_{i-1,i-2})=(a_{i,i-2}-a_{j-1,i-2})b_{i-1},
\end{equation}
\begin{equation}
i=3,\dots,s,\,\,\,\,\,j=2,\dots,i-2,\nonumber
\end{equation}
or

(II)
\begin{equation}
\label{eq_oc_2}
a_{i,j}(b_{j-1}-a_{j,j-1})=(a_{i,j-1}-a_{j,j-1})b_{j},
\end{equation}
\begin{equation}
i=3,\dots,s,\,\,\,\,\,j=2,\dots,i-1.\nonumber
\end{equation}
\end{thm}
\begin{proof}
Consider any equation from the set in Theorem~\ref{thm1}, Eq.~(\ref{eq_oc_2N_alpha}), at fixed $j=m$ and arbitrary (possibly equal) $p, q > m+1$ (so the corresponding coefficients are non-zero). After cross multiplication:
\begin{equation}
\alpha_{p,m+1}\alpha_{qm}=\alpha_{q,m+1}\alpha_{pm}.
\end{equation}
Summing on $q$ up to some $i-1$ one gets
\begin{equation}
\alpha_{p,m+1}\sum_{q=m+2}^{i-1}\alpha_{qm}=
\alpha_{pm}\sum_{q=m+2}^{i-1}\alpha_{q,m+1}.
\end{equation}
Using the inverse transformation (\ref{eq_a_from_alpha}) and after adding and subtracting $\alpha_{m+1,m}$ ($=a_{m+1,m}$) to the sum on the left:
\begin{equation}
\label{eq_th_alpha_a}
\alpha_{p,m+1}(a_{i-1,m}-a_{m+1,m})=\alpha_{pm}a_{i-1,m+1}.
\end{equation}
Next, summing Eq.~(\ref{eq_th_alpha_a}) on $p$ up to $i$:
\begin{equation}
\label{eq_th_alpha_sum_a}
\left(\sum_{p=m+2}^{i}\alpha_{p,m+1}\right)(a_{i-1,m}-a_{m+1,m})=
\left(\sum_{p=m+2}^{i}\alpha_{pm}\right)a_{i-1,m+1}.
\end{equation}
Again, after using Eq.~(\ref{eq_a_from_alpha}) and after adding and subtracting $\alpha_{m+1,m}$ ($=a_{m+1,m}$) to the sum on the right the result is:
\begin{equation}
a_{i,m+1}(a_{i-1,m}-a_{m+1,m})=
(a_{im}-a_{m+1,m})a_{i-1,m+1}.
\end{equation}
For convenience, set $j=m+1$, with $2\leqslant j<i-1$.
This gives Eq.~(\ref{eq_oc_1a}).

Alternatively, one can pick from the set (\ref{eq_oc_2N_alpha}) an equation that involves both $\alpha$ and $\beta$:
\begin{equation}
\label{eq_th_alpha_beta}
\alpha_{p,m+1}\beta_m=\alpha_{pm}\beta_{m+1}.
\end{equation}
Adding Eq.~(\ref{eq_th_alpha_a}) with $i=s+1$ to Eq.~(\ref{eq_th_alpha_beta}):
\begin{equation}
\alpha_{p,m+1}(\beta_m+a_{sm}-a_{m+1,m})=\alpha_{pm}(\beta_{m+1}+\alpha_{s,m+1}).
\end{equation}
After using Eq.~(\ref{eq_beta_b})
\begin{equation}
\alpha_{p,m+1}(b_m-a_{m+1,m})=\alpha_{pm}b_{m+1}.
\end{equation}
Now summing on $p$ up to $i$:
\begin{equation}
\left(\sum_{p=m+2}^{i}\alpha_{p,m+1}\right)(b_m-a_{m+1,m})=
\left(\sum_{p=m+2}^{i}\alpha_{pm}\right)b_{m+1}.
\end{equation}
Similarly to the transformation of Eq.~(\ref{eq_th_alpha_sum_a}) one gets
\begin{equation}
a_{i,m+1}(b_m-a_{m+1,m})=
(a_{im}-a_{m+1,m})b_{m+1}.
\end{equation}
As before, set $j=m+1$. This gives Eqs.~(\ref{eq_oc_1b}) and (\ref{eq_oc_2}) and completes the proof.
\end{proof}

A brief analysis of Eqs.~(\ref{eq_oc_1a})--(\ref{eq_oc_2}) is now in order. The form (II) (\ref{eq_oc_2}) involves both $a_{ij}$ and $b_i$ for all possible $i$ and $j$.
The form (I), (\ref{eq_oc_1a}) and (\ref{eq_oc_1b}), uses only $a_{ij}$ for most of the conditions, and only for $j=i-1$, where the coefficients $b_i$ are unavoidable, uses them. Eq.~(\ref{eq_oc_1b}) is obviously the same as (\ref{eq_oc_2}) apart from what index ranges are used.

A standard Runge-Kutta method has $s(s+1)/2$ coefficients (accounting also for relations (\ref{eq_c_from_a})). A 2N-storage method has $2s-1$ coefficients (since $A_1$=0). Thus, there must be $s(s+1)/2-(2s-1)=(s-1)(s-2)/2$ independent relations between the coefficients in the Butcher tableu if the method is a 2N-storage method.
This is precisely the number of independent conditions in the form (I) or (II).

There are two important remarks for using the form (I) or form (II) in practice.

\begin{rmk}
The conditions in the $\alpha$-form (\ref{eq_oc_2N_alpha}) are general, there are no special cases since no $\alpha_{ij}$, $\beta_i$ can be zero for a 2N-storage method with $A_i\neq0$, $i=2,\dots,s$.
The conditions in the polynomial form (\ref{eq_oc_1a})--(\ref{eq_oc_2}) may have spurious zero solutions. Therefore when solving the order conditions for $a_{ij}$, $b_i$ one should also evaluate $\alpha_{ij}$, $\beta_i$ and check that Eq.~(\ref{eq_oc_2N_alpha}) is also satisfied.
\end{rmk}

\begin{rmk}
While no $\alpha_{ij}$, $\beta_i$ can be zero, in contrast, some $a_{ij}$, $b_i$ can be zero or degenerate. These special situations need to be handled case by case simultaneously with the standard order conditions (\ref{eq_oc_RK_b})--(\ref{eq_oc_RK_bac}) (and the higher order ones, if a higher order method is being constructed). For example, consider Eq.~(\ref{eq_oc_1b}) for $i=4$:
\begin{equation}
\label{eq_a43_ex}
a_{43}(b_2-a_{32})=(a_{42}-a_{32})b_3.
\end{equation}
If $b_3=0$ (allowed for a 4-stage 2N-storage method), this automatically dictates $b_2=a_{32}$ because $a_{43}$ ($=\alpha_{43}$) cannot be zero. This means that Eq.~(\ref{eq_a43_ex}) is trivially satisfied and it is not as constraining as the other equations in the set. (Had $b_3$ not been zero, one could use Eq.~(\ref{eq_a43_ex}) to solve for $a_{42}$ in terms of $a_{43}$ and the rest.) Examples of how to deal with special cases for a (4,3) 2N-storage Runge-Kutta method are given in \ref{sec_app_spec43}.
\end{rmk}

\section{Relations between the coefficients}
\label{sec_relations}
The relations between the coefficients in the $A$-form and $a$-form will now be revisited taking into account the conditions in Eqs.~(\ref{eq_oc_1a})--(\ref{eq_oc_2}). The inverse relations in Eq.~(\ref{eq_Ai}) and (\ref{eq_Bi}) are recursive. While the recursion may be convenient for computation, it is useful to establish a direct mapping. Also, Eqs.~(\ref{eq_W_ab_to_A}) and (\ref{eq_W_ab0_to_A}) that appear in the literature in the same way as given by Williamson~\cite{WILLIAMSON198048} represent just one possible choice from a more general set that will now be presented.

\begin{thm}
The mapping of the coefficients of the $A$-form to the $a$-form is given by
\begin{eqnarray}
a_{ij} &=& \sum_{k=j+1}^{i} B_{k-1}\prod_{l=j+1}^{k-1}A_l,\label{eq_map_AB_to_a}\\
b_i &=& \sum_{k=i+1}^{s+1} B_{k-1}\prod_{l=i+1}^{k-1}A_l,\label{eq_map_AB_to_b}
\end{eqnarray}
and the inverse mapping from the $a$-form to the $A$-form by
\begin{eqnarray}
A_i &=& \frac{a_{k,i-1}-a_{k-1,i-1}}{a_{ki}-a_{k-1,i}}\,\,\,\,\,\mbox{for any $k$ such that }i<k\leqslant s,\label{eq_map_a_to_A}\\
&&\mbox{or}\nonumber\\
A_i &=& \frac{b_{i-1}-a_{s,i-1}}{b_{i}-a_{s,i}},\label{eq_map_ab_to_A}\\
&&\mbox{and}\nonumber\\
B_i &=& \left\{
\begin{array}{ll}
a_{i+1,i}, & i=1,\dots,s-1,\\
b_s, & i=s.
\end{array}
\right.\label{eq_map_ab_to_B}
\end{eqnarray}
\end{thm}
\begin{proof}
Eq.~(\ref{eq_map_AB_to_a}) follows from combining Eqs.~(\ref{eq_alpha_A}) and (\ref{eq_a_from_alpha}), Eq.~(\ref{eq_map_AB_to_b}) from Eqs.~(\ref{eq_alpha_A}), (\ref{eq_beta_B}) and (\ref{eq_b_from_beta}).

Substituting Eq.~(\ref{eq_alpha_a}) into Eq.~(\ref{eq_ratio_alpha}) produces
Eq.~(\ref{eq_map_a_to_A}), and substituting Eq.~(\ref{eq_beta_b}) into Eq.~(\ref{eq_ratio_beta}) produces Eq.~(\ref{eq_map_ab_to_A}).

Finally, Eq.~(\ref{eq_map_ab_to_B}) follows from using Eqs.~(\ref{eq_alpha_A}), (\ref{eq_beta_B}) and (\ref{eq_a_from_alpha}), (\ref{eq_b_from_beta}) for $\alpha_{i,i-1}$ and $\beta_s$. This completes the proof.
\end{proof}

Equations such as (\ref{eq_map_AB_to_a}) and (\ref{eq_map_AB_to_b}) appeared in the literature before for particular integration schemes. See, for instance, explicitly spelled out relations for $s=5$ in Eq.~(12) of Ref.~\cite{CK1994} or for $s=7$ in the Appendix of Ref.~\cite{ALLAMPALLI20093837}. Here they are given in full generality.

The mappings~(\ref{eq_map_AB_to_a})--(\ref{eq_map_ab_to_B}) need to be connected with the recursive forms (\ref{eq_W_ab_to_B})--(\ref{eq_Bi}). One can start with the mapping from the $A$-form to $a$-form.

Eq.~(\ref{eq_ratio_alpha}) for fixed $j$ and some $n>j+1$ gives
\begin{equation}
\alpha_{n,j}=A_{j+1}\alpha_{n,j+1}.
\end{equation}
Summing on $n$ up to $i$
\begin{equation}
\sum_{n=j+2}^i\alpha_{n,j}=A_{j+1}\sum_{n=j+2}^i\alpha_{n,j+1}.
\end{equation}
Similarly to the transformations in the proof of Theorem~\ref{th_oc_alpha},
\begin{equation}
\label{eq_aij_intm}
a_{i,j}-a_{j+1,j}=A_{j+1}a_{i,j+1}.
\end{equation}
As $a_{j+1,j}=B_j$, one gets Eq.~(\ref{eq_Ai}) for $j<i-1$ case.

Using Eq.~(\ref{eq_aij_intm}) for $j=i-1$ (and the fact that the method is explicit, $a_{i,i}\equiv0$) gives Eq.~(\ref{eq_Ai}) for $j=i-1$ case.

Now consider Eq.~(\ref{eq_ratio_beta}) at some $j$:
\begin{equation}
\label{eq_betaj_intm}
\beta_j=A_{j+1}\beta_{j+1}.
\end{equation}
Substituting Eq.~(\ref{eq_beta_b}) into~(\ref{eq_betaj_intm}) and rearranging produces
\begin{equation}
b_j=A_{j+1}b_{j+1}+(a_{sj}-A_{j+1}a_{s,j+1}).
\end{equation}
According to the first line of Eq.~(\ref{eq_Ai}) the term in the parentheses is simply $B_j$. This gives Eq.~(\ref{eq_Bi}), case $i<s$. Using Eq.~(\ref{eq_beta_B}) for $i=s$ and Eq.~(\ref{eq_beta_b}) gives Eq.~(\ref{eq_Bi}), case $i=s$.

Consider now a transformation from the $a$-form to $A$-form. Eq.~(\ref{eq_map_ab_to_B}) is the same as Eq.~(\ref{eq_W_ab_to_B}). Next, solve the first case $i<s$ of Eq.~(\ref{eq_Bi}) for $A_{i+1}$ and shift the indexing $i+1\to i$. This gives Eq.~(\ref{eq_W_ab_to_A}).

\subsection{Correct form of Eq.~(\ref{eq_W_ab0_to_A})}
\label{sec_Wcorr_b0}

As discussed earlier, while for a 2N-storage method with $A_i\neq0$, $i=2,\dots,s$ always $\beta_i\neq0$ for any $i=1,\dots,s$, $b_i$ can be zero (except for $b_s=\beta_s$). Eq.~(\ref{eq_W_ab_to_A}) is obviously not applicable and one needs to handle the special case. Any of Eqs.~(\ref{eq_map_a_to_A}), (\ref{eq_map_ab_to_A}) would do as there are no special cases -- if any of the numerators or denominators vanish, then the 2N-storage condition (\ref{eq_oc_2N_alpha}) is not satisfied and the Runge-Kutta method in the $a$-form is not a 2N-storage method with $A_i\neq0$, $i=2,\dots,s$. A form closest to the intended form of Eq.~(\ref{eq_W_ab0_to_A}) can however be introduced for comparison. Consider Eq.~(\ref{eq_map_a_to_A}) for $k=i+1$:
\begin{equation}
A_i=\frac{a_{i+1,i-1}-a_{i,i-1}}{a_{i+1,i}}.\label{eq_W_ab0_to_A_1}
\end{equation}
As $a_{i+1,i}=B_i$, the correct form of Eq.~(\ref{eq_W_ab0_to_A}) is then
\begin{equation}
A_i=\frac{a_{i+1,i-1}-a_{i,i-1}}{B_{i}},\,\,\,\,\,i\neq1.\label{eq_W_ab0_corr}\\
\end{equation}
Notice the difference, -- instead of $c_i$ in Eq.~(\ref{eq_W_ab0_to_A}) there is $a_{i,i-1}$ in Eq.~(\ref{eq_W_ab0_corr}). Since for $i=2$ $c_i=a_{i,i-1}$, Eq.~(\ref{eq_W_ab0_to_A}) is still correct for $b_2=0$, but it is incorrect for $b_i=0$, $i>2$.\footnote{Unless there are additional simplifications that enforce $c_i=a_{i,i-1}$, \textit{e.g.}, $a_{31}=0$, $c_3=a_{32}$.}

The proper form, Eq.~(\ref{eq_W_ab0_corr}), or equivalently (\ref{eq_W_ab0_to_A_1}), is always applicable, so one can use Eq.~(\ref{eq_W_ab0_corr}) instead of branching into $b_i=0$ and $b_i\neq0$ cases. The preferable form though is Eqs.~(\ref{eq_map_a_to_A})--(\ref{eq_map_ab_to_B}) (or (\ref{eq_W_ab0_to_A_1}) which is a particular choice in (\ref{eq_map_a_to_A})) as they follow directly from the 2N-storage order conditions~(\ref{eq_oc_2N_alpha}).

It is quite possible that the special case $b_i=0$ for $i=2$ was improperly generalized to $i>2$ in Ref.~\cite{WILLIAMSON198048}. Eq.~(\ref{eq_W_ab0_to_A}) appeared in multiple publications going back to 1980~\cite{WILLIAMSON198048}; why the error was not noticed until now deserves some discussion.

There are a number of possible reasons. First, the explicit 2N-storage conditions in the $\alpha$- (\ref{eq_oc_2N_alpha}) or $a$-form~(\ref{eq_oc_1a})--(\ref{eq_oc_2}) that allow one to straightforwardly derive the relations (\ref{eq_map_a_to_A})--(\ref{eq_map_ab_to_B}) were not available before.

Second, consider when Eq.~(\ref{eq_W_ab0_to_A}) could be used in practice. If there is a Runge-Kutta method in the standard $a$-form that is known to be a 2N-storage method, one would use Eqs.~(\ref{eq_W_ab_to_B})--(\ref{eq_W_ab0_to_A}) to convert it to the 2N-storage $A$-form. However, this is almost never the case, since standard explicit Runge-Kutta methods are not 2N-storage methods. One has to solve the nonlinear order conditions to develop such methods. Thus, as discussed in Sec.~\ref{sec_over_2n}, the mapping is almost exclusively used in the opposite direction, Eqs.~(\ref{eq_Ai}), (\ref{eq_Bi}): map the $A$-form to $a$-form, substitute into the standard order conditions, \textit{e.g.}, Eqs.~(\ref{eq_oc_RK_b})--(\ref{eq_oc_RK_bac}) and solve for $A_i$, $B_i$.

Third, the 2N-storage order conditions are much more constraining than the standard ones and the resulting systems of nonlinear equations are of higher order, especially when written in the $A$-form. For example, for a 5-stage method a simple term $b_2c_2$ unrolls into
$$
B_1B_5A_5A_4A_3+B_1B_4A_4A_3+B_1B_3A_3+B_1B_2.
$$

Fourth, as a consequence of the previous point, there are very few analytically solvable cases, the most prominent being $(3,3)$ (and, as is derived in Secs.~\ref{sec_lsrk43} and \ref{sec_lsrk53}, $(4,3)$ and $(5,3)$ methods). For $(3,3)$ 2N-storage methods always $b_3\neq0$, and the $b_2=0$ case is accidentally correct due to $c_2=a_{21}$.

Fifth, adding more stages to a method does not simplify, but, in fact, complicates the situation. In a standard Runge-Kutta method in the $a$-form the number of order conditions is constant and adding another stage to an $s$-stage method brings $s+1$ new additional degrees of freedom, while adding another stage to a 2N-storage method brings \textit{more} 2N-storage conditions and gains only two additional degrees of freedom.

Sixth, unless one deliberately imposes $b_i=0$, for $i=3$ (or higher, where applicable) condition, solves the constraints and checks the mappings (\ref{eq_W_ab_to_B})--(\ref{eq_Bi}) in both directions, the incorrectness of Eq.~(\ref{eq_W_ab0_to_A}) may never reveal itself.

\subsection{Examples of 2N-storage methods with $b_i=0$, $i>2$}
\label{sec_check_conv}

Here two coefficient schemes that numerically illustrate incorrectness of Eq.~(\ref{eq_W_ab0_to_A}) for $i>2$ are presented. The Butcher tableau for the $(4,3)$ method is shown in Table~\ref{tab_eqW_example43} and for $(5,3)$ method in Table~\ref{tab_eqW_example53}.

\begin{table}[h]
\centering
\hspace{-15mm}
\parbox{.4\linewidth}{
\[
\begin{array}{c|cccc}
0 & & & \\[1.5mm]
\frac{1}{2} & \frac{1}{2} & & & \\[1.5mm]
\frac{5}{9} & \frac{2}{9} & \frac{1}{3} & & \\[1.5mm]
\frac{3}{4} & \frac{3}{176} & \frac{51}{88} & \frac{27}{176} & \\[1.5mm]
\hline\\[-4mm]
& \frac{2}{9} & \frac{1}{3} & 0  & \frac{4}{9}
\end{array}
\]
}
\parbox{0.1\linewidth}{
\[
\begin{array}{c|c}
\phantom{-}0 & \frac{1}{2} \\[1.5mm]
-\frac{5}{6} & \frac{1}{3} \\[1.5mm]
\phantom{-}\frac{130}{81} & \frac{27}{176} \\[1.5mm]
-\frac{243}{704} & \frac{4}{9} \\[1.5mm]
\end{array}
\]
}
\caption{The Butcher tableau (left) and 2N-storage coefficients $A_i$, $B_i$ (right) for a (4,3) 2N-storage Runge-Kutta method with $b_3=0$.\label{tab_eqW_example43}}
\end{table}

\begin{table}[h]
\centering
\hspace{-20mm}
\parbox{.53\linewidth}{
\[
\begin{array}{c|ccccc}
0 & & & & \\[1.5mm]
\frac{1}{3} & \phantom{-}\frac{1}{3} & & & & \\[1.5mm]
\frac{1}{2} & \phantom{-}\frac{1}{8} & \phantom{-}\frac{3}{8} & & & \\[1.5mm]
\frac{7}{9} & \phantom{-}\frac{1}{18} & \phantom{-}\frac{1}{2} & \phantom{-}\frac{2}{9} & & \\[1.5mm]
1 & \phantom{-}\frac{81}{328} & \phantom{-}\frac{51}{328} & -\frac{16}{41} & \phantom{-}\frac{81}{82} & \\[1.5mm]
\hline\\[-4mm]
& \phantom{-}\frac{1}{18} & \phantom{-}\frac{1}{2} & \phantom{-}\frac{2}{9} & \phantom{-}0  & \phantom{-}\frac{2}{9}
\end{array}
\]
}
\parbox{0.05\linewidth}{
\[
\begin{array}{c|c}
\phantom{-}0 & \frac{1}{3} \\[1.5mm]
-\frac{5}{9} & \frac{3}{8} \\[1.5mm]
\phantom{-}\frac{9}{16} & \frac{2}{9} \\[1.5mm]
-\frac{452}{729} & \frac{81}{82} \\[1.5mm]
-\frac{729}{164} & \frac{2}{9} \\[1.5mm]
\end{array}
\]
}
\caption{The Butcher tableau (left) and 2N-storage coefficients $A_i$, $B_i$ (right) for a (5,3) 2N-storage Runge-Kutta method with $b_4=0$.\label{tab_eqW_example53}}
\end{table}

The reader can verify that the coefficients of these Runge-Kutta methods satisfy the order conditions~(\ref{eq_oc_RK_b})--(\ref{eq_oc_RK_bac}). In the right parts of these tables the 2N-storage coefficients, computed using Eqs.~(\ref{eq_map_ab_to_A}) and (\ref{eq_map_ab_to_B}) are also shown. When the latter are converted back to $a_{ij}$, $b_i$ coefficients, the left parts of the tables are reproduced. However, if the original conversion formulas, Eqs.~(\ref{eq_W_ab_to_B})--(\ref{eq_W_ab0_to_A}) are used, one gets different sets of $A_i$, $B_i$, and as a result, different Butcher tableaux when $A_i$, $B_i$ are converted to $a_{ij}$, $b_i$. The incorrect sets of coefficients that result from using Eq.~(\ref{eq_W_ab0_to_A}) are illustrated in Tables~\ref{tab_eqW_example43_not} and \ref{tab_eqW_example53_not}. They no longer satisfy the order conditions consistent with the third order of global accuracy. A Matlab script illustrating the conversions for the two (4,3) and (5,3) schemes described here and one extra (5,3) scheme with $b_3=0$ is included in \ref{sec_app_matlab}.

\begin{table}[h]
\centering
\parbox{0.1\linewidth}{
\[
\begin{array}{c|c}
\phantom{-}0 & \frac{1}{2} \\[1.5mm]
-\frac{5}{6} & \frac{1}{3} \\[1.5mm]
\phantom{-}\frac{38}{243} & \frac{27}{176} \\[1.5mm]
-\frac{243}{704} & \frac{4}{9} \\[1.5mm]
\end{array}
\]
}\hspace{5mm}
\parbox{.4\linewidth}{
\[
\begin{array}{c|cccc}
0 & & & \\[1.5mm]
\frac{1}{2} & \frac{1}{2} & & & \\[1.5mm]
\frac{5}{9} & \frac{2}{9} & \frac{1}{3} & & \\[1.5mm]
\frac{77}{108} & \frac{961}{4752} & \frac{283}{792} & \frac{27}{176} & \\[1.5mm]
\hline\\[-4mm]
& \frac{2}{9} & \frac{1}{3} & 0  & \frac{4}{9}
\end{array}
\]
}
\caption{An \textbf{incorrect} set of 2N-storage coefficients $A_i$, $B_i$ (left) and improper Butcher tableau that result from using the incorrect special case, Eq.~(\ref{eq_W_ab0_to_A}) for the (4,3) low-storage Runge-Kutta method shown in Table~\ref{tab_eqW_example43}.\label{tab_eqW_example43_not}}
\end{table}

\begin{table}[h]
\centering
\parbox{0.05\linewidth}{
\[
\begin{array}{c|c}
\phantom{-}0 & \frac{1}{3} \\[1.5mm]
-\frac{5}{9} & \frac{3}{8} \\[1.5mm]
\phantom{-}\frac{9}{16} & \frac{2}{9} \\[1.5mm]
-\frac{862}{729} & \frac{81}{82} \\[1.5mm]
-\frac{729}{164} & \frac{2}{9} \\[1.5mm]
\end{array}
\]
}\hspace{10mm}
\parbox{.53\linewidth}{
\[
\begin{array}{c|ccccc}
0 & & & & \\[1.5mm]
\frac{1}{3} & \phantom{-}\frac{1}{3} & & & & \\[1.5mm]
\frac{1}{2} & \phantom{-}\frac{1}{8} & \phantom{-}\frac{3}{8} & & & \\[1.5mm]
\frac{7}{9} & \phantom{-}\frac{1}{18} & \phantom{-}\frac{1}{2} & \phantom{-}\frac{2}{9} & & \\[1.5mm]
\frac{11}{36} & \phantom{-}\frac{2483}{5904 } & -\frac{103}{656} & -\frac{349}{369} & \phantom{-}\frac{81}{82} & \\[1.5mm]
\hline\\[-4mm]
& \phantom{-}\frac{1}{18} & \phantom{-}\frac{1}{2} & \phantom{-}\frac{2}{9} & \phantom{-}0  & \phantom{-}\frac{2}{9}
\end{array}
\]
}
\caption{An \textbf{incorrect} set of 2N-storage coefficients $A_i$, $B_i$ (left) and improper Butcher tableau that result from using the incorrect special case, Eq.~(\ref{eq_W_ab0_to_A}) for the (5,3) low-storage Runge-Kutta method shown in Table~\ref{tab_eqW_example53}.\label{tab_eqW_example53_not}}
\end{table}

\section{Expressing $a_{ij}$ and $A_i$, $B_i$ through $b_i$, $c_i$ for a 2N-storage method}
\label{sec_abc}
An $s$-stage explicit Runge-Kutta method has $s-1$ non-trivial coefficients $c_i$. If they are considered free parameters, then the conditions (\ref{eq_c_from_a}) allow for eliminating $s-1$ coefficients $a_{ij}$ (for instance, $a_{i,1}$) as linearly dependent. Then there are $s(s-1)/2-(s-1)=(s-1)(s-2)/2$ remaining independent $a_{ij}$. This number is the same as the number of conditions in form (I) (\ref{eq_oc_1a}), (\ref{eq_oc_1b}) or (II) (\ref{eq_oc_2}). Also, the number of $b_i$, $c_i$ is $2s-1$, the same as the number of non-trivial $A_i$, $B_i$. This hints that it may be possible to consider $b_i$, $c_i$ as independent parameters and solve the conditions (\ref{eq_oc_1a})--(\ref{eq_oc_2}), \textit{i.e.}, express $a_{ij}$ in terms of $b_i$, $c_i$ for a 2N-storage method. Another hint is given by, \textit{e.g.}, Eq.~(\ref{eq_oc_2}). Notice that it relates the coefficients at stage $i$ to the ones at the previous stages since $j<i$. If those are known, then Eq.~(\ref{eq_oc_2}) is a system of linear equations for the coefficients at stage $i$.

\begin{thm}
In the general case, given that corresponding denominators do not vanish, the $a_{ij}$ coefficients of the $a$-form of a 2N-storage method at stage $i$ can be expressed as:
\begin{equation}
\label{eq_a_bc_rec}
a_{ij}=\frac{b_j}{\displaystyle\sum_{k=1}^j b_k-c_j}\left[c_i-c_j-\sum_{k=j+1}^{i-1}a_{ik}\right].
\end{equation}
\end{thm}
\begin{proof}
Equation~(\ref{eq_a_bc_rec}) will be proved by two nested inductions, referred as \textit{outer} (progressing from stage to stage) and \textit{inner} (within the stage).

To start the outer induction consider the first non-trivial stage $i=2$. With $c_1=0$, from Eq.~(\ref{eq_a_bc_rec})
\begin{equation}
a_{21}=\frac{b_1}{b_1-c_1}(c_2-c_1)=c_2.
\end{equation}
(The sum on $k$ in Eq.~(\ref{eq_a_bc_rec}) is zero.) Thus, Eq.~(\ref{eq_a_bc_rec}) is satisfied for the lowest $i=2$.

Consider now stage $i$ and assume that all stages before $i$ satisfy Eq.~(\ref{eq_a_bc_rec}).

To start the inner induction, consider the first condition at stage $i$ in the form (II) (\ref{eq_oc_2}), $j=2$:
\begin{equation}
a_{i,2}(b_1-a_{21})=(a_{i,1}-a_{21})b_2.
\end{equation}
Assume a general case where no denominators vanish and $b_i\neq0$ for all $i$. Then, using Eq.~(\ref{eq_c_from_a}) to solve for $a_{i,1}$ and using $a_{21}=c_2$:
\begin{equation}
c_i-\sum_{j=2}^{i-1}a_{ij}-c_2=\frac{b_1-c_2}{b_2}a_{i,2}.
\end{equation}
Solving for $a_{i,2}$ gives
\begin{equation}
a_{i,2}=\frac{b_2}{b_1+b_2-c_2}\left[c_i-c_2-\sum_{j=3}^{i-1}a_{ij}\right]
\end{equation}
meaning that Eq.~(\ref{eq_a_bc_rec}) is valid for that coefficient.

Note that although $a_{i,1}$ are eliminated as dependent variables, Eq.~(\ref{eq_a_bc_rec}) is formally valid for them as well:
\begin{equation}
a_{i,1}=\frac{b_1}{b_1-c_1}\left[c_i-c_1-\sum_{j=2}^{i-1}a_{ij}\right]=c_i-\sum_{j=2}^{i-1}a_{ij}.
\end{equation}
Now consider the second, inner induction: assume that at stage $i$ for $k<j$ Eq.~(\ref{eq_a_bc_rec}) holds for $a_{ik}$.
Consider $j$-th condition (\ref{eq_oc_2}). Substitute $a_{i,j-1}$ (assumed valid in the inner induction) and $a_{j,j-1}$ (outer induction). First,
\begin{equation}
\label{eq_aij1}
a_{i,j-1}-a_{j,j-1}=\frac{b_{j-1}}{\displaystyle \sum_{k=1}^{j-1}b_k-c_{j-1}}
\left[c_i-c_j-\sum_{k=j}^{i-1}a_{ik}\right].
\end{equation}
Second,
\begin{equation}
\label{eq_bj1_ajj1}
b_{j-1}-a_{j,j-1}=b_{j-1}\,\frac{\displaystyle \sum_{k=1}^{j-1}b_k-c_{j}}{\displaystyle \sum_{k=1}^{j-1}b_k-c_{j-1}}.
\end{equation}
Then substituting Eqs.~(\ref{eq_aij1}) and (\ref{eq_bj1_ajj1}) into (\ref{eq_oc_2}) and rearranging produces
\begin{equation}
a_{ij}\left(\sum_{k=1}^{j-1}b_k-c_{j}\right)=
b_j\left(c_i-c_j-\sum_{k=j+1}^{i-1}a_{ik}\right)-b_ja_{ij}.
\end{equation}
Solving for $a_{ij}$ produces Eq.~(\ref{eq_a_bc_rec}). This proves the inner induction, as validity for $k<j$ was assumed and validity for $k=j$ was proved. And this also proves the outer induction, as validity for stages before stage $i$ was assumed and validity for the stage $i$ was proved.
\end{proof}

\begin{rmk}
Eq.~(\ref{eq_a_bc_rec}) is recursive. One first evaluates $a_{i,i-1}$, then $a_{i,i-2}$, etc. A more complicated, but non-recursive relation is also available after recursively substituting Eq.~(\ref{eq_a_bc_rec}) into itself:
\begin{equation}
\label{eq_a_bc}
a_{ij}=b_j\sum_{k=1}^{i-j}\frac{\displaystyle \prod_{l=1}^{k-1}\left(\sum_{n=1}^{j+l-1}b_n-c_{j+l}\right)}{\displaystyle \prod_{m=0}^{k-1}\left(\sum_{n=1}^{j+m}b_n-c_{j+m}\right)}(c_{j+k}-c_{j+k-1}).
\end{equation}
Eq.~(\ref{eq_a_bc}) does not seem to be more practically beneficial than Eq.~(\ref{eq_a_bc_rec}), and the proof is therefore omitted.
\end{rmk}

\begin{thm}
In the general case, given that corresponding denominators do not vanish and with convention $b_0=0$, $c_{s+1}=1$, the coefficients $A_i$, $B_i$ of the $A$-form of a 2N-storage method can be expressed as
\begin{eqnarray}
A_i &=& \frac{b_{i-1}}{b_i}\,\frac{\displaystyle \sum_{k=1}^{i-1}b_k-c_{i}}{\displaystyle \sum_{k=1}^{i-1}b_k-c_{i-1}},\label{eq_A_bc}\\
B_i &=& \frac{b_i}{\displaystyle\sum_{k=1}^i b_k-c_i}(c_{i+1}-c_i),\label{eq_B_bc}\\
i &=& 1,\dots,s.
\end{eqnarray}
\end{thm}
\begin{proof}
Eq.~(\ref{eq_B_bc}) follows from using the expression for $a_{ij}$ (\ref{eq_a_bc_rec}) in (\ref{eq_map_ab_to_B}), first line. To accommodate $B_s$, the second line of Eq.~(\ref{eq_map_ab_to_B}) is used: with $c_{s+1}=1$ it also gives Eq.~(\ref{eq_B_bc}) for $i=s$.

Next, Eq.~(\ref{eq_B_bc}) is substituted into Eq.~(\ref{eq_W_ab_to_A}):
\begin{equation}
A_i=\frac{1}{b_i}\left[b_{i-1}-\frac{b_{i-1}}{\displaystyle\sum_{k=1}^{i-1} b_k-c_i}(c_{i}-c_{i-1})\right].
\end{equation}
Rearranging gives Eq.~(\ref{eq_A_bc}).
\end{proof}

\begin{rmk}
As a cross check, Eqs.~(\ref{eq_A_bc}), (\ref{eq_B_bc}) provide an alternative route to Eq.~(\ref{eq_a_bc}) -- one substitutes Eqs.~(\ref{eq_A_bc}), (\ref{eq_B_bc}) into Eq.~(\ref{eq_map_AB_to_a}).
\end{rmk}

\subsection{Special case $A_i=-1$, $i>1$}

Relations such as Eqs.~(\ref{eq_a_bc_rec}), (\ref{eq_A_bc}) and (\ref{eq_B_bc}) are applicable in the general case. Special cases, such as $b_i=0$, $c_2=c_3$, $c_3=c_4$, etc. need to be treated case-by-case. There is, however, a standalone special case that will now be considered. It has been noticed in the literature that setting all $A_i=-1$, $i>1$ simplifies the order conditions in the $A$-form and solutions are possible if there are enough stages in the method to accommodate that choice. Thanks to Eq.~(\ref{eq_oc_2N_alpha}), that case is also solvable. In the general case one, uses $c_i$ and $b_i$ as free parameters. As now $s-1$ parameters are fixed by $A_{i>1}=-1$, one has to give up $c_i$ and keep $b_i$ as the free parameters. Thus, the Butcher tableau needs to be expressed completely in terms of $b_i$.

If some $A_{j+1}=-1$, then from Eq.~(\ref{eq_ratio_alpha}) it follows that
\begin{equation}
\alpha_{ij}=-\alpha_{i,j+1}.\nonumber
\end{equation}
Substituting the relation (\ref{eq_alpha_a}) and rearranging,
\begin{equation}
\label{eq_aij_aij1}
a_{ij}+a_{i,j+1}=a_{i-1,j}+a_{i-1,j+1}.
\end{equation}
Similarly, using the corresponding relations for $\beta_i$,
\begin{equation}
\label{eq_bj_bj1}
b_j+b_{j+1}=a_{s,j}+a_{s,j+1}.
\end{equation}
Equation~(\ref{eq_aij_aij1}) is applicable for $i=j+1,\dots,s$. Together with Eq.~(\ref{eq_bj_bj1}) it means that the sum of elements in columns $j$ and $j+1$ for any row $i$ is the same and equal to
\begin{equation}
\label{eq_aij1_bj1}
a_{ij}+a_{i,j+1}=b_j+b_{j+1}.
\end{equation}
These are the simplified 2N-storage constraints for a single $A_{j+1}=-1$. If now all $A_{i>1}=-1$, this applies to all pairs of adjacent columns in the Butcher tableau. Due to its triangular structure, all elements can now be expressed through the weights $b_i$.

One can start at an arbitrary row $i$. For the last non-zero element, following from Eq.~(\ref{eq_aij1_bj1})
\begin{equation}
\label{eq_aii1}
a_{i,i-1}=b_{i-1}+b_i.
\end{equation}
For the element to the left
\begin{equation}
\label{eq_aii2}
a_{i,i-2}+a_{i,i-1}=b_{i-2}+b_{i-1}.
\end{equation}
Subtracting Eq.~(\ref{eq_aii1}) from (\ref{eq_aii2}) one gets
\begin{equation}
a_{i,i-2}=b_{i-2}-b_i.\nonumber
\end{equation}
It is clear by induction that for any $i>j$
\begin{equation}
\label{eq_aij_from_bi}
a_{ij}=b_j+(-1)^{i-j+1}b_i.
\end{equation}
Using Eq.~(\ref{eq_c_from_a}),
\begin{equation}
c_i=\sum_{j=1}^{i-1}a_{ij}=\sum_{j=1}^{i-1}\left(b_j+(-1)^{i-j+1}b_i\right)=
\sum_{j=1}^{i-1}b_j+(-1)^{i+1}b_i\sum_{j=1}^{i-1}(-1)^j.\nonumber
\end{equation}
The sum
\begin{equation}
\sum_{j=1}^{i-1}(-1)^j=\frac{1}{2}\left((-1)^{i-1}-1\right)\nonumber
\end{equation}
is either $-1$ (for even $i$) or $0$ (for odd $i$). Then
\begin{equation}
\label{eq_ci_from_bi}
c_i=\sum_{j=1}^{i-1}b_j+\frac{1}{2}\left(1+(-1)^{i}\right)b_i.
\end{equation}
If written in a recursive form, Eq.~(\ref{eq_ci_from_bi}) matches Eq.~(35) of Ref.~\cite{BERNARDINI20094182}.
Equation~(\ref{eq_ci_from_bi}) encodes a simple pattern:
\begin{eqnarray}
c_2&=&b_1+b_2,\nonumber\\
c_3&=&b_1+b_2,\nonumber\\
c_4&=&b_1+b_2+b_3+b_4,\nonumber\\
c_5&=&b_1+b_2+b_3+b_4,\nonumber\\
&\phantom{=}&\dots\nonumber
\end{eqnarray}
For even number of stages $s$ inevitably $c_s=1$ due to Eq.~(\ref{eq_oc_RK_b}). Using Eqs.~(\ref{eq_aij_from_bi}) and (\ref{eq_ci_from_bi}) one rewrites the standard order conditions, \textit{e.g.}, Eqs.~(\ref{eq_oc_RK_b})--(\ref{eq_oc_RK_bac}), in terms of the weights $b_i$.

\section{General solution for $(4,3)$ 2N-storage Runge-Kutta methods}
\label{sec_lsrk43}
The results developed in the previous sections will now be applied to a $(4,3)$ 2N-storage method. Several schemes with rational coefficients were presented in Ref.~\cite{CK1994}, but not the general solution. Here no extra constraints (\textit{e.g.}, to enforce a particular stability region) are applied, but the general solution to the system of Eqs.~(\ref{eq_oc_RK_b})--(\ref{eq_oc_RK_bac}) and (\ref{eq_oc_1a})--(\ref{eq_oc_2}) that results in a four-stage third-order 2N-storage method is presented.

Using Eq.~(\ref{eq_a_bc_rec}), one eliminates $a_{ij}$ and the resulting system of equations are the Eqs.~(\ref{eq_oc_RK_b})--(\ref{eq_oc_RK_bac}) with $b_i$, $c_i$ only. This is a system of four equations with seven variables $c_2$, $c_3$, $c_4$, $b_1$, $b_2$, $b_3$ and $b_4$. It is beneficial to take $c_i$ as the three free parameters and solve for $b_i$. It turns out that when one progressively eliminates the variables and reduces the system to a single equation, that equation is of third order. However, in the case of $b_2$ the free term is absent, so it is effectively of second order. (And the special cases $b_i=0$ are handled in \ref{sec_app_spec43}.) Thus, the goal is to eliminate $b_1$, $b_3$ and $b_4$ first. $b_1$ is immediately given by Eq.~(\ref{eq_oc_RK_b}):
\begin{equation}
\label{eq_43_b1}
b_1=1-b_2-b_3-b_4.
\end{equation}
Solving Eqs.~(\ref{eq_oc_RK_bc}) and (\ref{eq_oc_RK_bc2}) for $b_3$ and $b_4$ gives
\begin{eqnarray}
b_3 &=& \frac{\displaystyle c_4/2-1/3}{\displaystyle c_3(c_4-c_3)}
-\frac{\displaystyle c_2(c_4-c_2)}{\displaystyle c_3(c_4-c_3)}\,b_2,\label{eq_43_b3}\\
b_4 &=& \frac{\displaystyle 1/3-c_3/2}{\displaystyle c_4(c_4-c_3)}
+\frac{\displaystyle c_2(c_3-c_2)}{\displaystyle c_4(c_4-c_3)}\,b_2.\label{eq_43_b4}
\end{eqnarray}
The last (and the most elaborate) step is to substitute $b_1$, $b_3$ and $b_4$ into Eq.~(\ref{eq_a_bc_rec}) for $a_{43}$, $a_{42}$ and $a_{32}$ and then substitute all variables into Eq.~(\ref{eq_oc_RK_bac}). After tedious but straightforward algebra one arrives at the following result:
\begin{equation}
\label{eq_43_x}
\frac{1}{6}z_1(z_2-z_3)+\left[z_1z_4+z_3z_5+\frac{1}{6}(z_2-z_1)(c_3-c_2)\right]x
+(c_3-c_2)\left[z_4-z_5-(c_4-c_2)(z_1+z_3)\right]x^2=0,
\end{equation}
where
\begin{eqnarray}
x &\equiv& b_2c_2,\\
z_1 &\equiv& c_3c_4(1-c_2)-\frac{1}{2}(c_3+c_4)+\frac{1}{3},\\
z_2 &\equiv& c_3+c_4-\frac{3}{2}c_3c_4-\frac{2}{3},\\
z_3 &\equiv& c_4(c_4-c_3)(1-c_3)+\frac{1}{2}c_3-\frac{1}{3},\\
z_4 &\equiv& c_3c_4 - \frac{1}{2}c_2(c_3+c_4)-\frac{1}{3}(c_3+c_4-2c_2),\\
z_5 &\equiv& \frac{1}{2}(c_3c_4-c_2c_3-c_2c_4)+\frac{1}{2}c_2-\frac{1}{6}(c_3+c_4).
\end{eqnarray}
Solving the quadratic Eq.~(\ref{eq_43_x}) with chosen values of $c_2\neq c_3\neq c_4$ produces $b_2$ and the backward substitution into Eqs.~(\ref{eq_43_b3}), (\ref{eq_43_b4}), (\ref{eq_43_b1}) and (\ref{eq_a_bc_rec}) produces the rest of the Butcher tableau. The $a$-form is then easily converted to the $A$-form with Eqs.~(\ref{eq_map_a_to_A})--(\ref{eq_map_ab_to_B}).

The special cases need to be handled individually and they are solved for a (4,3) method in \ref{sec_app_spec43}.

Reference~\cite{CK1994} argued that a nearly optimal $(4,3)$ method is achieved when
\begin{equation}
\label{eq_43_extra}
\sum_{ijk}b_ia_{ij}a_{jk}c_k=\frac{1}{24}
\end{equation}
is additionally imposed (where the left-hand side for a four-stage method is simply $b_4a_{43}a_{32}c_2$) and presented five (4,3) schemes satisfying this constraint. This also makes the resulting $(4,3)$ scheme fourth-order accurate for linear problems. Some experimentation with the Groebner basis reduction shows that with the extra constraint the resulting final equation for chosen $b_i$ is of a higher order than before. (Sixth order was achieved with the monomial orderings that were tried.) However, since the general solution is now available, if one is looking for schemes with rational coefficients, it is easier to search for rational solutions of Eq.~(\ref{eq_43_x}) first, and then pick those that satisfy Eq.~(\ref{eq_43_extra}). A few additional schemes have been found in the course of this work. Four of them are listed in Tables~\ref{tab_43_set1} and \ref{tab_43_set2}. More schemes are available, however, the illustration here is limited to the ones where the nodes $c_i$ are in increasing order and the weights $b_i\geqslant-3/8$.

\begin{table}[h]
\centering
\hspace{-5mm}
\parbox{.35\linewidth}{
\[
\begin{array}{c|cccc}
0 & & & \\[1.5mm]
\frac{1}{4} & \phantom{-}\frac{1}{4} & & & \\[1.5mm]
\frac{7}{12} & -\frac{1}{12} & \phantom{-}\frac{2}{3} & & \\[1.5mm]
\frac{4}{5} & \phantom{-}\frac{12}{25} & -\frac{23}{50} & \phantom{-}\frac{39}{50} & \\[1.5mm]
\hline\\[-4mm]
& \phantom{-}\frac{1}{6} & \phantom{-}\frac{1}{6} & \phantom{-}\frac{9}{26} & \phantom{-}\frac{25}{78}
\end{array}
\]
}
\parbox{0.1\linewidth}{
\[
\begin{array}{c|c}
\phantom{-}0 & \frac{1}{4} \\[1.5mm]
-\frac{1}{2} & \frac{2}{3} \\[1.5mm]
-\frac{13}{9} & \frac{39}{50} \\[1.5mm]
-\frac{846}{625} & \frac{25}{78} \\[1.5mm]
\end{array}
\]
}\hspace{0.07\linewidth}
\parbox{.35\linewidth}{
\[
\begin{array}{c|cccc}
0 & & & \\[1.5mm]
\frac{1}{5} & \phantom{-}\frac{1}{5} & & & \\[1.5mm]
\frac{3}{5} & -\frac{3}{20} & \phantom{-}\frac{3}{4} & & \\[1.5mm]
\frac{13}{15} & \phantom{-}\frac{143}{540} & -\frac{5}{36} & \phantom{-}\frac{20}{27} & \\[1.5mm]
\hline\\[-4mm]
& -\frac{1}{9} & \phantom{-}\frac{2}{3} & \phantom{-}\frac{5}{72} & \phantom{-}\frac{3}{8}
\end{array}
\]
}
\parbox{0.1\linewidth}{
\[
\begin{array}{c|c}
\phantom{-}0 & \frac{1}{5} \\[1.5mm]
-\frac{7}{15} & \frac{3}{4} \\[1.5mm]
-\frac{6}{5} & \frac{20}{27} \\[1.5mm]
-\frac{145}{81} & \frac{3}{8} \\[1.5mm]
\end{array}
\]
}
\caption{The Butcher tableaux and 2N-storage coefficients $A_i$, $B_i$ for two (4,3) 2N-storage Runge-Kutta methods with rational coefficients that are fourth-order accurate for linear problems, \textit{i.e.}, additionally satisfy Eq.~(\ref{eq_43_extra}), labeled (4,3)$_1$ and (4,3)$_2$, respectively.\label{tab_43_set1}}
\end{table}

\begin{table}[h]
\centering
\hspace{-5mm}
\parbox{.35\linewidth}{
\[
\begin{array}{c|cccc}
0 & & & \\[1.5mm]
\frac{2}{15} & \phantom{-}\frac{2}{15} & & & \\[1.5mm]
\frac{2}{5} & -\frac{7}{20} & \phantom{-}\frac{3}{4} & & \\[1.5mm]
\frac{4}{5} & \phantom{-}\frac{169}{180} & -\frac{5}{4} & \phantom{-}\frac{10}{9} & \\[1.5mm]
\hline\\[-4mm]
& \phantom{-}\frac{3}{8} & -\frac{3}{8} & \phantom{-}\frac{5}{8} & \phantom{-}\frac{3}{8}
\end{array}
\]
}
\parbox{0.1\linewidth}{
\[
\begin{array}{c|c}
\phantom{-}0 & \frac{2}{15} \\[1.5mm]
-\frac{29}{45} & \frac{3}{4} \\[1.5mm]
-\frac{9}{5} & \frac{10}{9} \\[1.5mm]
-\frac{35}{27} & \frac{3}{8} \\[1.5mm]
\end{array}
\]
}\hspace{0.07\linewidth}
\parbox{.35\linewidth}{
\[
\begin{array}{c|cccc}
0 & & & \\[1.5mm]
\frac{13}{28} & \phantom{-}\frac{13}{28} & & & \\[1.5mm]
\frac{4}{7} & -\frac{32}{91} & \phantom{-}\frac{12}{13} & & \\[1.5mm]
\frac{37}{42} & \phantom{-}\frac{1091}{2184} & -\frac{14}{351} & \phantom{-}\frac{91}{216} & \\[1.5mm]
\hline\\[-4mm]
& \phantom{-}\frac{5}{26} & \phantom{-}\frac{4}{13} & \phantom{-}\frac{7}{26} & \phantom{-}\frac{3}{13}
\end{array}
\]
}
\parbox{0.1\linewidth}{
\[
\begin{array}{c|c}
\phantom{-}0 & \frac{13}{28} \\[1.5mm]
-\frac{99}{112} & \frac{12}{13} \\[1.5mm]
-\frac{16}{7} & \frac{91}{216} \\[1.5mm]
-\frac{427}{648} & \frac{3}{13} \\[1.5mm]
\end{array}
\]
}
\caption{The Butcher tableaux and 2N-storage coefficients $A_i$, $B_i$ for two more (4,3) 2N-storage Runge-Kutta methods with rational coefficients that are fourth-order accurate for linear problems, \textit{i.e.}, additionally satisfy Eq.~(\ref{eq_43_extra}), labeled (4,3)$_3$ and (4,3)$_4$, respectively.\label{tab_43_set2}}
\end{table}

\section{General solution for $(5,3)$ 2N-storage Runge-Kutta methods}
\label{sec_lsrk53}
A general solution for (5,3) 2N-storage Runge-Kutta method will now be presented. In this case, there are five free parameters and it is convenient to choose $c_2$, $c_3$, $c_4$, $c_5$, and $b_5$ as such. Other choices result in the final equation of higher order than presented here. The strategy is similar to the (4,3) case: use the mapping in Eq.~(\ref{eq_a_bc_rec}) to express all $a_{ij}$ parameters in terms of $c_i$ and $b_i$, then eliminate $b_1$, $b_3$ and $b_4$ and the final equation for $b_2$ is of second order. The chain of intermediate variables presented below follows the chain of substitutions and no attempt was made to further simplify the expressions. It is assumed that the encountered denominators do not vanish. The first set of intermediate variables is
\begin{equation}
\omega_1=\frac{c_4(1/2-b_5c_5)-(1/3-b_5c_5^2)}{c_3(c_4-c_3)},\,\,\,\,\,
\rho_1=-\frac{c_2(c_4-c_2)}{c_3(c_4-c_3)},
\end{equation}
\begin{equation}
\omega_2=\frac{1/3-b_5c_5^2-c_3(1/2-b_5c_5)}{c_4(c_4-c_3)},\,\,\,\,\,
\rho_2=\frac{c_2(c_3-c_2)}{c_4(c_4-c_3)},
\end{equation}
\begin{equation}
\omega_3=\frac{(c_5-c_4)\omega_2}{1-b_5-c_4},\,\,\,\,\,
\rho_3=\frac{(c_5-c_4)\rho_2}{1-b_5-c_4},
\end{equation}
\begin{equation}
\omega_4=b_5c_4\omega_3-1/6,\,\,\,\,\,
\rho_4=b_5c_4\rho_3,
\end{equation}
\begin{equation}
\omega_5=b_5c_3(c_5-c_3-\omega_3),\,\,\,\,\,
\rho_5=-b_5c_3\rho_3,
\end{equation}
\begin{equation}
\omega_6=1-b_5-c_3-\omega_2,\,\,\,\,\,\rho_6=-\rho_2,
\end{equation}
\begin{equation}
\omega_7=b_5c_2[(c_5-\omega_3)(\omega_6-\omega_1)-c_2\omega_6+c_3\omega_1],\,\,\,\,\,\zeta_7=b_5c_2\rho_3(\rho_1-\rho_6),
\end{equation}
\begin{equation}
\rho_7=b_5c_2[\rho_3(\omega_1-\omega_6)+(c_5-\omega_3)(\rho_6-\rho_1)-c_2\rho_6+c_3\rho_1],
\end{equation}
\begin{equation}
\omega_8=c_2[(c_4-c_2)\omega_6-(c_4-c_3)\omega_1],\,\,\,\,\,
\rho_8=c_2[(c_4-c_2)\rho_6-(c_4-c_3)\rho_1],
\end{equation}
\begin{equation}
\omega_9=1-b_5-c_2-\omega_1-\omega_2,\,\,\,\,\,
\rho_9=-\rho_1-\rho_2.
\end{equation}

The second set of intermediate variables is
\begin{eqnarray}
\chi_1^0&=&\omega_4\omega_6\omega_9,\nonumber\\
\chi_1^1&=&\omega_4\omega_6\rho_9+\omega_4\rho_6\omega_9+\rho_4\omega_6\omega_9,\nonumber\\
\chi_1^2&=&\omega_4\rho_6\rho_9+\rho_4\omega_6\rho_9+\rho_4\rho_6\omega_9,\nonumber\\
\chi_1^3&=&\rho_4\rho_6\rho_9,\nonumber\\
\chi_2^0&=&\omega_1\omega_5\omega_9,\nonumber\\
\chi_2^1&=&\omega_1\omega_5\rho_9+\omega_1\rho_5\omega_9+\rho_1\omega_5\omega_9,\nonumber\\
\chi_2^2&=&\omega_1\rho_5\rho_9+\rho_1\omega_5\rho_9+\rho_1\rho_5\omega_9,\nonumber\\
\chi_2^3&=&\rho_1\rho_5\rho_9,\nonumber\\
\chi_3^0&=&0,\nonumber\\
\chi_3^1&=&\omega_7,\nonumber\\
\chi_3^2&=&\rho_7,\nonumber\\
\chi_3^3&=&\zeta_7,\nonumber\\
\chi_4^0&=&c_3(c_4-c_3)\omega_1\omega_2\omega_9,\nonumber\\
\chi_4^1&=&c_3(c_4-c_3)(\omega_1\omega_2\rho_9+\omega_1\rho_2\omega_9+\rho_1\omega_2\omega_9),\nonumber\\
\chi_4^2&=&c_3(c_4-c_3)(\omega_1\rho_2\rho_9+\rho_1\omega_2\rho_9+\rho_1\rho_2\omega_9),\nonumber\\
\chi_4^3&=&c_3(c_4-c_3)\rho_1\rho_2\rho_9,\nonumber\\
\chi_5^0&=&0,\nonumber\\
\chi_5^1&=&\omega_2\omega_8,\nonumber\\
\chi_5^2&=&\omega_2\rho_8+\rho_2\omega_8,\nonumber\\
\chi_5^3&=&\rho_2\rho_8,\nonumber\\
\chi_6^0&=&0,\nonumber\\
\chi_6^1&=&c_2(c_3-c_2)\omega_1\omega_6,\nonumber\\
\chi_6^2&=&c_2(c_3-c_2)(\omega_1\rho_6+\rho_1\omega_6),\nonumber\\
\chi_6^3&=&c_2(c_3-c_2)\rho_1\rho_6.
\end{eqnarray}
Then the final coefficients are
\begin{equation}
{\cal C}_k=\sum_{n=1}^6\chi_n^k,\,\,\,k=0,\dots3
\end{equation}
and the equation for $b_2$ is
\begin{equation}
{\cal C}_3b_2^3+{\cal C}_2b_2^2+{\cal C}_1b_2+{\cal C}_0=0.
\end{equation}
Direct substitution shows that ${\cal C}_3=0$, so effectively the equation for $b_2$ is of second order. After solving for $b_2$,
\begin{eqnarray}
b_3&=&\omega_1+\rho_1b_2,\\
b_4&=&\omega_2+\rho_2b_2,\\
b_1&=&1-b_2-b_3-b_4-b_5,
\end{eqnarray}
and the rest of the Butcher tableau is constructed using Eq.~(\ref{eq_a_bc_rec}).

There seems to be no extra benefit in solving special cases for the (5,3) method, therefore they are not pursued further. In general, the special cases $b_i=0$ lead to linear equations that can be derived with the same strategies as the ones explained in \ref{sec_app_spec43} for the (4,3) methods. In fact, the (5,3) method in Table~\ref{tab_eqW_example53} of Sec.~\ref{sec_Wcorr_b0} illustrating the error in Eq.~(\ref{eq_W_ab0_to_A}) was derived that way.

It may be that the trend continues and a 2N-storage method with third order of global accuracy always leads to the final equation for $b_2$ (or $b_3$, or $b_4$) which is of the order not larger than three. The reason is that for an $s$-stage method there are $2s-1$ parameters $b_i$, $c_i$ and four constraints, Eqs.~(\ref{eq_oc_RK_b})--(\ref{eq_oc_RK_bac}) (once $a_{ij}$ are expressed through $b_i$, $c_i$ with the help of Eq.~(\ref{eq_a_bc_rec})). This means that independently of $s$, one can choose $b_1$, $b_2$, $b_3$ and $b_4$ as the variables to solve for. The way they enter in the numerators and denominators is similar to the (4,3) and (5,3) methods and thus, when all terms are brought to the common denominator, the final equation is again of the order not more than three.
It may also be that the process of deriving that final equation can be automated with symbolic manipulation software, given the regular structure of the $a_{ij}$ coefficients exhibited in Eq.~(\ref{eq_a_bc_rec}).

\section{Numerical experiments}
\label{sec_num}
If one is interested in developing a 2N-storage method of a chosen order for a problem at hand, the recommendation is to solve the nonlinear algebraic system of 2N-storage order conditions: the standard order conditions, Eqs.~(\ref{eq_oc_RK_b})--(\ref{eq_oc_RK_bac}) (and the higher order ones, if a higher-order method is sought) and the 2N-storage ones, Eqs.~(\ref{eq_oc_1a}) and (\ref{eq_oc_1b}) or (\ref{eq_oc_2}) together with whatever additional constraints are desirable. This is easier to do in the original Butcher tableau parameters rather than in the 2N-storage parameters $A_i$, $B_i$.

\begin{table}[h]
\centering
\parbox{.6\linewidth}{
\centering
\[
\begin{array}{c|ccccc}
0 & & & & \\[1.5mm]
\frac{1}{4} & \phantom{-}\frac{1}{4} & & & & \\[1.5mm]
\frac{8}{15} & -\frac{16}{225} & \phantom{-}\frac{136}{225} & & & \\[1.5mm]
\frac{12}{17} & \phantom{-}\frac{832}{1005} & -\frac{18584}{17085} & \phantom{-}\frac{1100}{1139} & & \\[1.5mm]
\frac{5}{6} & -\frac{13213}{60300} & \phantom{-}\frac{13312}{15075} & -\frac{1875}{11792} & \phantom{-}\frac{289}{880} & \\[1.5mm]
\hline\\[-4mm]
& \phantom{-}\frac{15}{94} & \phantom{-}\frac{8}{47} & \phantom{-}\frac{1025}{4136 } & \phantom{-}\frac{867}{4136}  & \phantom{-}\frac{10}{47}
\end{array}
\]
}
\parbox{0.15\linewidth}{
\centering
\[
\begin{array}{c|c}
\phantom{-}0 & \frac{1}{4} \\[1.5mm]
-\frac{17}{32} & \frac{136}{225} \\[1.5mm]
-\frac{9856}{5625} & \frac{1100}{1139} \\[1.5mm]
-\frac{1127375}{329171} & \frac{289}{880} \\[1.5mm]
-\frac{4913}{8800} & \frac{10}{47} \\[1.5mm]
\end{array}
\]
}
\caption{The Butcher tableau and the 2N-storage coefficients for the (5,3)$_1$ method that produces the lowest error across the three test problems.\label{tab_53_1}}
\end{table}
\begin{table}[h]
\centering
\parbox{.6\linewidth}{
\centering
\[
\begin{array}{c|ccccc}
0 & & & & \\[1.5mm]
\frac{1}{4} & \phantom{-}\frac{1}{4} & & & & \\[1.5mm]
\frac{4}{7} & -\frac{8}{49} & \phantom{-}\frac{36}{49} & & & \\[1.5mm]
\frac{2}{3} & \phantom{-}\frac{163}{2394} & \phantom{-}\frac{3484}{10773} & \phantom{-}\frac{847}{3078} & & \\[1.5mm]
\frac{13}{14} & \phantom{-}\frac{12053}{11172} & \phantom{-}\frac{2960}{25137} & \phantom{-}\frac{847}{2052} & \phantom{-}\frac{3}{14} & \\[1.5mm]
\hline\\[-4mm]
& \phantom{-}\frac{37}{258} & \phantom{-}\frac{220}{1161} & \phantom{-}\frac{847}{2322} & \phantom{-}\frac{6}{43}  & \phantom{-}\frac{7}{43}
\end{array}
\]
}
\parbox{0.15\linewidth}{
\centering
\[
\begin{array}{c|c}
\phantom{-}0 & \frac{1}{4} \\[1.5mm]
-\frac{9}{16} & \frac{36}{49} \\[1.5mm]
-\frac{62032}{41503} & \frac{847}{3078} \\[1.5mm]
\phantom{-}\frac{5929}{9234} & \frac{3}{14} \\[1.5mm]
-\frac{45}{98} & \frac{7}{43} \\[1.5mm]
\end{array}
\]
}
\caption{The Butcher tableau and the 2N-storage coefficients for the (5,3)$_2$ method that has one of the largest stability regions.\label{tab_53_2}}
\end{table}

As closed-form expressions have been derived for (4,3) and (5,3) schemes, several schemes with rational coefficients are presented here, solely for illustration purposes. A simple grid search relatively quickly revealed over two million of (5,3) methods with rational coefficients and no attempt was made to find ``the best'' scheme over that set. What is optimal really depends on the application. The four (5,3) schemes presented here simply illustrate the possibilities. If one prefers schemes with rational coefficients, for (4,3) and (5,3) (and possibly ($s>5$,3)) schemes, one can generate rational solutions and then filter out the ones with desired additional characteristics.

To compare the methods, three test problems are considered:
\begin{enumerate}
\item $y'=y\cos(x)$, $y(0)=1$, $0\leqslant x\leqslant20$, with the solution $y_{exact}(x)=\exp(\sin(x))$.
\item $y'=4y\sin^3(x)\cos(x)$, $y(0)=1$, $0\leqslant x\leqslant20$, with the solution $y_{exact}(x)=\exp(\sin^4(x))$.
\item $y'=-y^{3/2}/2$, $y(0)=1$, $0\leqslant x\leqslant20$, with the solution $y_{exact}(x)=1/\sqrt{1+x}$.
\end{enumerate}
The first two are from Ref.~\cite{CK1994} and the third, with an enlarged interval, from Ref.~\cite{BBBproc2004}.

\begin{table}[h]
\centering
\parbox{.6\linewidth}{
\centering
\[
\begin{array}{c|ccccc}
0 & & & & \\[1.5mm]
\frac{2}{9} & \phantom{-}\frac{2}{9} & & & & \\[1.5mm]
\frac{1}{2} & -\frac{1}{8} & \phantom{-}\frac{5}{8} & & & \\[1.5mm]
\frac{13}{18} & \phantom{-}\frac{179}{360} & -\frac{99}{200} & \phantom{-}\frac{18}{25} & & \\[1.5mm]
\frac{9}{10} & \phantom{-}\frac{99}{1000} & \phantom{-}\frac{1109}{5000} & \phantom{-}\frac{162}{625} & \phantom{-}\frac{8}{25} & \\[1.5mm]
\hline\\[-4mm]
& \phantom{-}\frac{1}{6} & \phantom{-}\frac{1}{10} & \phantom{-}\frac{27}{80} & \phantom{-}\frac{17}{64}  & \phantom{-}\frac{25}{192}
\end{array}
\]
}
\parbox{0.15\linewidth}{
\centering
\[
\begin{array}{c|c}
\phantom{-}0 & \frac{2}{9} \\[1.5mm]
-\frac{5}{9} & \frac{5}{8} \\[1.5mm]
-\frac{14}{9} & \frac{18}{25} \\[1.5mm]
-\frac{36}{25} & \frac{8}{25} \\[1.5mm]
-\frac{8}{25} & \frac{25}{192} \\[1.5mm]
\end{array}
\]
}
\caption{The Butcher tableau and the 2N-storage coefficients for the (5,3)$_3$ method that is of fourth order for linear problems and has a somewhat extended stability region.\label{tab_53_3}}
\end{table}

\begin{table}[h]
\centering
\parbox{.6\linewidth}{
\centering
\[
\begin{array}{c|ccccc}
0 & & & & \\[1.5mm]
\frac{1}{4} & \phantom{-}\frac{1}{4} & & & & \\[1.5mm]
\frac{1}{2} & -\frac{1}{6} & \phantom{-}\frac{2}{3} & & & \\[1.5mm]
\frac{3}{4} & \phantom{-}\frac{1}{4} & \phantom{-}0 & \phantom{-}\frac{1}{2} & & \\[1.5mm]
1 & \phantom{-}0 & \phantom{-}\frac{2}{5} & \phantom{-}\frac{1}{5} & \phantom{-}\frac{2}{5} & \\[1.5mm]
\hline\\[-4mm]
& \phantom{-}\frac{1}{9} & \phantom{-}\frac{2}{9} & \phantom{-}\frac{1}{3} & \phantom{-}\frac{2}{9}  & \phantom{-}\frac{1}{9}
\end{array}
\]
}
\parbox{0.15\linewidth}{
\centering
\[
\begin{array}{c|c}
\phantom{-}0 & \frac{1}{4} \\[1.5mm]
-\frac{5}{8} & \frac{2}{3} \\[1.5mm]
-\frac{4}{3} & \frac{1}{2} \\[1.5mm]
-\frac{3}{4} & \frac{2}{5} \\[1.5mm]
-\frac{8}{5} & \frac{1}{9} \\[1.5mm]
\end{array}
\]
}
\caption{The Butcher tableau and the 2N-storage coefficients for the (5,3)$_4$ method whose stage before the last one is second-order accurate, and thus the method can be used as an embedded 3,2 pair for adaptive step size integration.\label{tab_53_4}}
\end{table}

The three test problems are integrated with the five (4,3) schemes of Ref.~\cite{CK1994}, the new (4,3)$_1$ scheme shown in Table~\ref{tab_43_set1} (left) and the new (5,3)$_i$ $i=1,\dots,4$ schemes shown in Tables~\ref{tab_53_1}--\ref{tab_53_4}. The error is defined as the distance between the exact and numerical solution, as function of the step size $h$:
\begin{equation}
d(h)=|y(x=20,h)-y_{exact}(x=20)|.
\end{equation}

The results for the ten methods are shown in Figs.~\ref{fig_test12} and \ref{fig_test3}. The results for the five (4,3) schemes of Ref.~\cite{CK1994} are shown in gray to give a rough idea of their performance as a baseline, and the results from the five new schemes are shown in color. It happens that the new (4,3)$_1$ scheme performs slightly better than the schemes of Ref.~\cite{CK1994} across the three problems. This is, however, a byproduct of a grid search over rational coefficients, rather than a deliberate attempt to optimize a (4,3) scheme for these specific test problems. For the (5,3) methods some search over a small subset of the uncovered two million schemes has been done to illustrate the following points. If one is concerned with optimizing a (5,3) method across these three test problems, the (5,3)$_1$ scheme in Table~\ref{tab_53_1} may be a reasonable choice. If one prefers to use the ``tall trees'' (linear) order conditions to manipulate the stability region (as is often the goal of introducing additional stages, beyond what is required to achieve a given global order of accuracy), (5,3)$_2$ scheme shown in Table~\ref{tab_53_2} has one of the larger stability regions, Fig.~\ref{fig_test3} (right), on the other hand, its performance in terms of the error $d(h)$ for the three test problems is unimpressive. If one prefers a (5,3) scheme which is fourth-order accurate for linear problems, \textit{i.e.}, satisfies additionally Eq.~(\ref{eq_43_extra}), and has a reasonable stability region, the (5,3)$_3$ scheme shown in Table~\ref{tab_53_3} may be a choice. If one is interested in computationally cheap adaptive step size integrator, the (5,3)$_4$ scheme in Table~\ref{tab_53_4} has additional properties:
\begin{eqnarray}
a_{51}+a_{52}+a_{53}+a_{54}&=&1,\nonumber\\
a_{52}c_2+a_{53}c_3+a_{54}c_4&=&\frac{1}{2}.\nonumber
\end{eqnarray}
In other words, the stage before the last stage provides (for free) a second-order estimate, and thus the method can be used as an embedded 3,2 pair where no extra computation is needed.

The stability regions, defined in the standard way~\cite{ButcherBook} for the new (4,3)$_1$ and (5,3)$_i$, $i=1,\dots,4$ schemes are shown in Fig.~\ref{fig_test3} (right).

\begin{figure}
\centering
\includegraphics[width=0.49\textwidth]{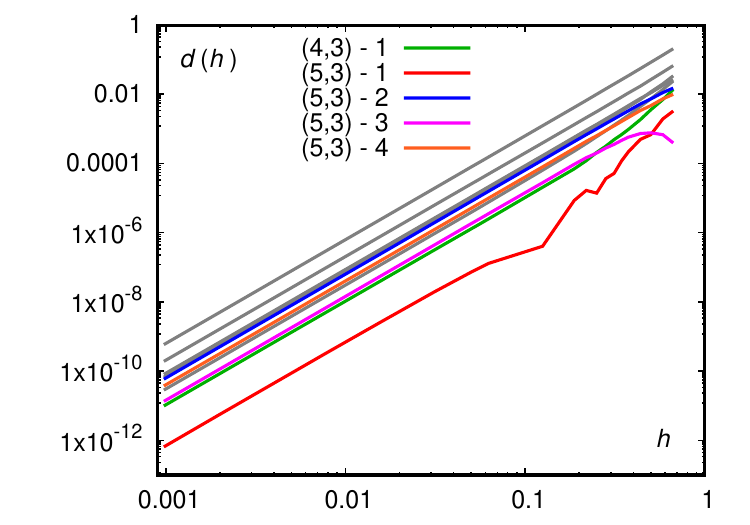}
\hfill
\includegraphics[width=0.49\textwidth]{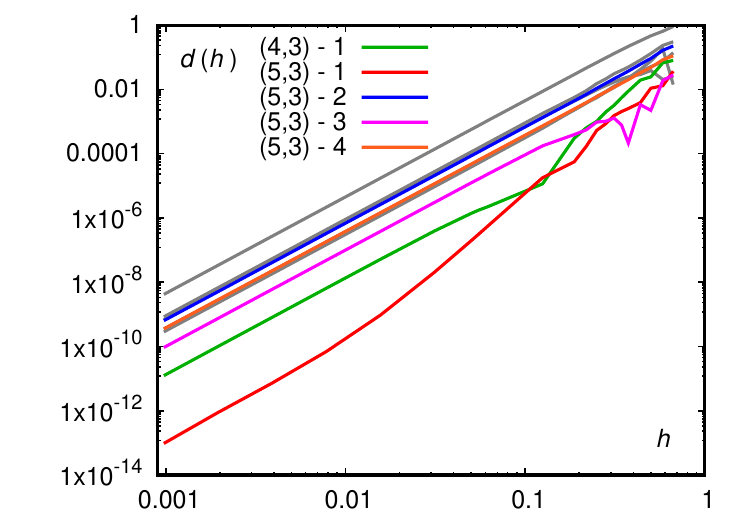}
\caption{Scaling of the new (4,3)$_1$ and (5,3)$_i$, $i=1,\dots4$ methods, shown in color, and the five (4,3) methods of Ref.~\cite{CK1994}, shown in gray, for test problems 1 (left) and 2 (right).\label{fig_test12}
}
\end{figure}
\begin{figure}
\centering
\includegraphics[width=0.49\textwidth]{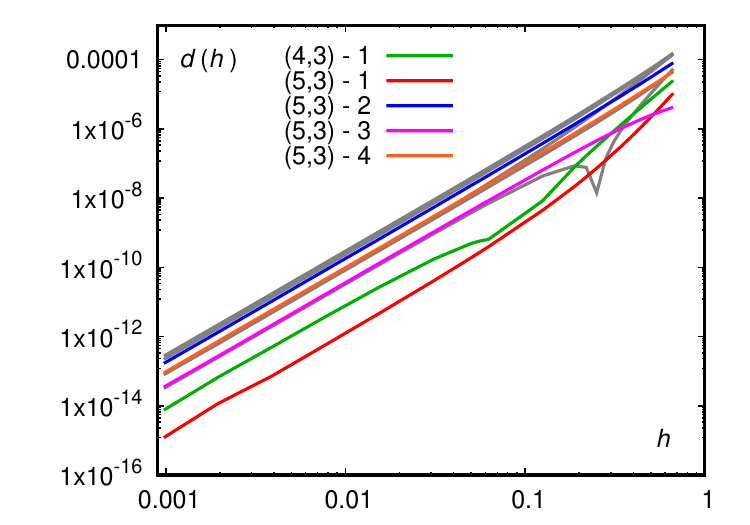}
\hfill
\includegraphics[width=0.49\textwidth]{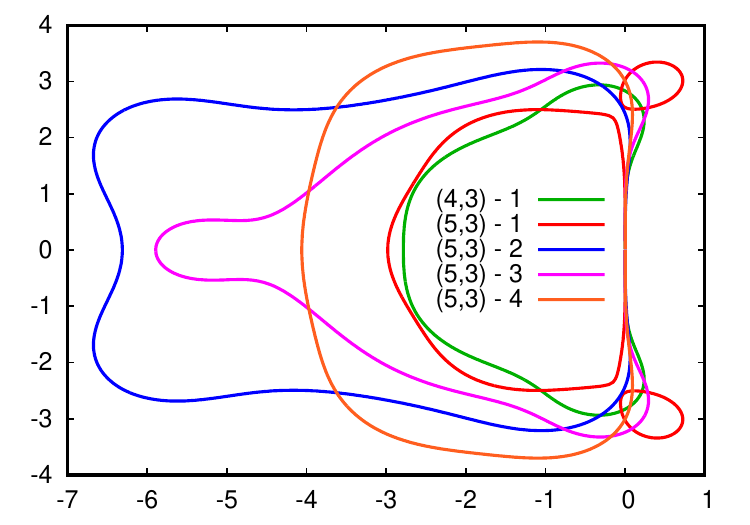}
\caption{Scaling of the new (4,3)$_1$ and (5,3)$_i$, $i=1,\dots4$ methods, shown in color, and the five (4,3) methods of Ref.~\cite{CK1994}, shown in gray, for test problem 3 (left). Stability regions for all (4,3) methods (identical due to satisfying Eq.~(\ref{eq_43_extra})) and (5,3) methods considered (right).\label{fig_test3}
}
\end{figure}

\section{Conclusion}
\label{sec_concl}
It appears that 2N-storage Runge-Kutta methods of Williamson~\cite{WILLIAMSON198048} possess some interesting structure that was explored in this work. The main results are the explicit 2N-storage conditions in the $\alpha$-form, Eq.~(\ref{eq_oc_2N_alpha}), and the $a$-form, Eqs.~(\ref{eq_oc_1a})--(\ref{eq_oc_2}), relations between the $a$-, $\alpha$- and $A$-forms of the method, Eqs.~(\ref{eq_a_from_alpha})--(\ref{eq_beta_b}), (\ref{eq_alpha_A}), (\ref{eq_beta_B}) and (\ref{eq_map_AB_to_a})--(\ref{eq_map_ab_to_B}). (The $a$-, $\alpha$- and $A$-form of a 2N-storage Runge-Kutta method are defined in Sec.~\ref{sec_over_2n}.) Next, all $a_{ij}$ parameters of the Butcher tableau of a 2N-storage method, independently of the order of global accuracy, can be expressed through the nodes $c_i$ and weights $b_i$, either recursively, Eq.~(\ref{eq_a_bc_rec}), or directly, Eq.~(\ref{eq_a_bc}). Having the explicit 2N-storage conditions spelled out for the first time uncovered an error in the special case of the Williamson formula~\cite{WILLIAMSON198048}. The suggestion is to use Eqs.~(\ref{eq_map_AB_to_a})--(\ref{eq_map_ab_to_B}) (or Eq.~(\ref{eq_W_ab0_to_A_1})) rather than the original formulas of Williamson, as the new relations are more transparent and do not require branching into special cases: if any of the encountered denominators vanish, the method is not a 2N-storage method anyway. Next, the 2N-storage coefficients $A_i$, $B_i$ can also be completely expressed through $b_i$ and $c_i$ as well, see Eqs.~(\ref{eq_A_bc}), (\ref{eq_B_bc}). Armed with these new relations, the full set of 2N-storage order conditions (including the standard ones at third order) is solved for (4,3) and (5,3) methods. The solution is presented in the form of a quadratic equation for $b_2$ (in both cases), whose coefficients are functions of free parameters. From the solution for $b_2$ full Butcher tableau for (4,3) and (5,3) 2N-storage methods is constructed with the relations presented here. For methods of order four and higher one has to resort to numerical solutions. It appears to be an easier task in the $a$-form, \textit{i.e.}, solving the standard order conditions together with the relations in Eqs.~(\ref{eq_oc_1a}), (\ref{eq_oc_1b}) or (\ref{eq_oc_2}). The validity of the solutions is illustrated with several (4,3) and (5,3) schemes in Sec.~\ref{sec_num}.

These new developments may bring the matters a step closer to understanding a remarkable feature of the 2N-storage Runge-Kutta methods of Williamson type: they appear to be automatically Lie group integrators, as examined in Ref.~\cite{Bazavov2021}, \textit{i.e.}, for an equation of the form
\begin{equation}
\frac{dY}{dt}=A(t,Y)Y,
\end{equation}
where $A(t,Y)$ is a matrix and $Y$ is a matrix or a vector, a simple modification of Eqs.~(\ref{eq_2N_W_dyi})--(\ref{eq_2N_W_i}) to
\begin{eqnarray}
\Delta Y_i &=& A_{i}\Delta Y_{i-1}+hA(t+c_{i}h,Y_{i-1}),\label{eq_2N_Lie_dyi}\\
Y_i &=&  \exp( B_{i}\Delta Y_i) Y_{i-1},\\
i &=&1,\dots,s.\label{eq_2N_Lie_i}
\end{eqnarray}
produces a valid geometric integrator with the same order of accuracy as the original classical 2N-storage scheme.

In fact, the classical (6,4) 2N-storage Runge-Kutta scheme of Ref.~\cite{BERLAND20061459} works remarkably well as a Lie group integrator for integration of the SU(3) gradient flow in lattice gauge theory, as concluded by Ref.~\cite{BazavovChuna2021} where several schemes were compared. For that reason the coefficients of that scheme were determined as a byproduct of this work with more significant digits than were given in the original publication~\cite{BERLAND20061459} (12 significant digits), allowing for applications in full or higher than the standard 64-bit floating point precision. This is achieved quite easily by solving the eight standard fourth-order order conditions together with the 2N-storage constraints~(\ref{eq_oc_1a}), (\ref{eq_oc_1b}) or (\ref{eq_oc_2}) with a standard Newton-Raphson method~\cite{NumRecBook}. The only feature worth mentioning is the use of the MPFR library~\cite{MPFR} for arbitrary precision arithmetic (1,000-bit in this case). For the interested reader the coefficients are given in \ref{sec_app_bbb} with 42-digit precision.

The new relations presented in this work uncover an interesting structure of the 2N-storage methods and some of their useful properties. Further investigation is warranted to understand if similar relations can be derived for other types of low-storage (or, perhaps, more general) Runge-Kutta methods.

\bigskip
\textbf{Acknowledgements.}
The author expresses deep gratitude to the CERN Theory Division for hospitality and financial support where a large fraction of this work was carried out during the sabbatical leave, and, in particular, to Matteo Di Carlo, Felix Erben, Philippe de Forcrand, Andreas J\"{u}ttner, Simon Kuberski and Tobias Tsang. The author is indebted to Steven Gottlieb and Leon Hostetler for careful reading and comments on the manuscript.
This work was in part supported
by the U.S. National Science Foundation under award
PHY23-09946.

\appendix

\section{Butcher tableau for an explicit Runge-Kutta method}
\label{sec_app_btab}
\begin{table}[h]
\centering
\hspace{-15mm}
\parbox{.4\linewidth}{
\[
\begin{array}{c|cccc}
0   &  & & & \\[1.5mm]
c_2 & a_{21} & & & \\[1.5mm]
c_3 & a_{31} & a_{32} & & \\[1.5mm]
c_4 & a_{41} & a_{42} & a_{43} & \\[1.5mm]
\hline\\[-4mm]
& b_1 & b_2 & b_3  & b_4
\end{array}
\]
}
\parbox{0.1\linewidth}{
\[
\begin{array}{c|c}
A_1 & B_1 \\[1.5mm]
A_2 & B_2 \\[1.5mm]
A_3 & B_3 \\[1.5mm]
A_4 & B_4 \\[1.5mm]
\end{array}
\]
}
\caption{A standard form of the Butcher tableau for a four-stage explicit Runge-Kutta method (left) and a corresponding table with the 2N-storage coefficients (right).\label{tab_bt_4}}
\end{table}

Butcher tableau for a four-stage explicit Runge-Kutta method is shown in the left part of Table~\ref{tab_bt_4}. For an explicit method the first row contains only zeros (\textit{i.e.}, $c_1=0$) and is often omitted. The coefficients $c_i$ are often called \textit{nodes} and $b_i$ \textit{weights} of the Runge-Kutta method. A set of 2N-storage coefficients $A_i$, $B_i$ can also be arranged vertically, as shown in the right part of Table~\ref{tab_bt_4}. For an explicit method always $A_1=0$.

\section{Order conditions for a Runge-Kutta method up to order 3}
\label{sec_app_oc}
The order conditions for a Runge-Kutta method to be globally of third order of accuracy are:
\begin{eqnarray}
\sum_i b_i &=&1,\label{eq_oc_RK_b}\\
\sum_i b_ic_i &=&\frac{1}{2},\label{eq_oc_RK_bc}\\
\sum_i b_ic_i^2 &=&\frac{1}{3},\label{eq_oc_RK_bc2}\\
\sum_i\sum_j b_i a_{ij} c_j &=&\frac{1}{6}.\label{eq_oc_RK_bac}
\end{eqnarray}

It is understood that in summations in Eqs.~(\ref{eq_oc_RK_b})--(\ref{eq_oc_RK_bac}) the indices run over all possible values where the corresponding coefficient is not zero.

\section{Special cases for the (4,3) low-storage Runge-Kutta methods}
\label{sec_app_spec43}
As an illustration, some special cases when Eq.~(\ref{eq_a_bc_rec}) may not be applicable to all $a_{ij}$ are presented. It is convenient to explicitly spell out the conditions (\ref{eq_oc_1a})--(\ref{eq_oc_2}) for the $(4,3)$ method. Form (I):
\begin{eqnarray}
a_{32}(b_1-a_{21}) &=& (a_{31}-a_{21})b_2,\label{eq_oc_43_f1_1}\\
a_{42}(a_{31}-a_{21}) &=&(a_{41}-a_{21})a_{32},\label{eq_oc_43_f1_2}\\
a_{43}(b_2-a_{32}) &=& (a_{42}-a_{32})b_3.\label{eq_oc_43_f1_3}
\end{eqnarray}
Form (II):
\begin{eqnarray}
a_{42}(b_1-a_{21}) &=&(a_{41}-a_{21})b_{2}.\label{eq_oc_43_f2_2}
\end{eqnarray}
The other two equations of the form (II) are the same as Eq.~(\ref{eq_oc_43_f1_1}) and (\ref{eq_oc_43_f1_3}).

\subsubsection{Case $b_2=0$, $b_3\neq0$}

From Eq.~(\ref{eq_oc_43_f1_1}) with $b_2=0$ it follows $b_1=a_{21}=c_2$ as $a_{32}$ cannot be zero. $c_2$ and $c_3$ are chosen as free parameters. Eqs.~(\ref{eq_oc_RK_b})--(\ref{eq_oc_RK_bc2}) are solved for $b_3$, $b_4$ and $c_4$:
\begin{eqnarray}
c_4 &=& \frac{1/3-c_3/2}{1/2-(1-c_2)c_3},\,\,\,\,\,1/2-(1-c_2)c_3\neq0,\label{eq_43_b20_c4}\\
b_4 &=& \frac{1/3-c_3/2}{c_4(c_4-c_3)},\,\,\,\,\,c_3\neq c_4,\label{eq_43_b20_b4}\\
b_3 &=& 1 - c_2 - b_4.\label{eq_43_b20_b3}
\end{eqnarray}
From Eq.~(\ref{eq_oc_43_f1_2})
\begin{equation}
a_{42}=\frac{c_4-c_2-a_{43}}{c_3-c_2}\,a_{32},\,\,\,\,\,c_2\neq c_3.\label{eq_43_b20_a42}
\end{equation}
From Eq.~(\ref{eq_oc_43_f1_3}) (since $b_3\neq0$ the same result is achieved from using the general case (\ref{eq_a_bc}))
\begin{equation}
a_{43}=\frac{b_3}{b_3-c_3+c_2}(c_4-c_3).\label{eq_43_b20_a43}
\end{equation}
Then from using Eq.~(\ref{eq_oc_RK_bac})
\begin{equation}
a_{32}=\frac{c_3-c_2}{c_2}\,\frac{1/6-b_4a_{43}c_3}{1/2-c_2(1-c_2)-b_4a_{43}}.\label{eq_43_b20_a32}
\end{equation}
To populate the Butcher tableau one chooses particular values for $c_2$ and $c_3$ and then proceeds in the order (\ref{eq_43_b20_c4}), (\ref{eq_43_b20_b4}), (\ref{eq_43_b20_b3}), (\ref{eq_43_b20_a43}), (\ref{eq_43_b20_a32}) and (\ref{eq_43_b20_a42}).

\subsubsection{Case $b_3=0$, $b_2\neq0$}

From Eq.~(\ref{eq_oc_43_f1_3}) $a_{32}=b_2$ as $a_{43}\neq0$. From Eq.~(\ref{eq_oc_43_f1_1}) $a_{31}=b_1$ which means $b_1+b_2=c_3$, then
\begin{equation}
b_1=c_3-b_2.\label{eq_43_b30_b1}
\end{equation}
 From Eq.~(\ref{eq_oc_RK_b})
\begin{equation}
c_3=1-b_4.\label{eq_43_b30_c3}
\end{equation}
$c_2$ and $c_4$ are chosen as free parameters and Eqs.~(\ref{eq_oc_RK_bc}), (\ref{eq_oc_RK_bc2}) are solved for $b_2$ and $b_4$:
\begin{eqnarray}
b_2&=&\frac{c_4/2-1/3}{c_2(c_4-c_2)},\,\,\,\,\,c_2\neq c_4,\label{eq_43_b30_b2}\\
b_4&=&\frac{1/3-c_2/2}{c_4(c_4-c_2)}.\label{eq_43_b30_b4}
\end{eqnarray}
From Eq.~(\ref{eq_oc_43_f1_2}) (since $b_2\neq0$ the same result is achieved from using the general case (\ref{eq_a_bc_rec}))
\begin{equation}
a_{42}=\frac{b_2}{c_3-c_2}(c_4-c_2-a_{43}),\,\,\,\,\,c_2\neq c_3.\label{eq_43_b30_a42}
\end{equation}
Then from Eq.~(\ref{eq_oc_RK_bac})
\begin{equation}
a_{43}=\frac{(c_3-c_2)/(6b_4)-b_2c_2(c_4-c_2)}{(c_3-c_2)c_3-b_2c_2}.\label{eq_43_b30_a43}
\end{equation}
After choosing $c_2$ and $c_4$ proceed in the order (\ref{eq_43_b30_b2}), (\ref{eq_43_b30_b4}), (\ref{eq_43_b30_c3}), (\ref{eq_43_b30_b1}), (\ref{eq_43_b30_a43}) and (\ref{eq_43_b30_a42}).

\subsubsection{Case $c_2=c_3$}

$c_2\neq c_4$ are taken as free parameters. From Eq.~(\ref{eq_oc_43_f1_1})
\begin{equation}
b_1=c_2-b_2\label{eq_43_c2c3_b1}
\end{equation}
and thus from Eq.~(\ref{eq_oc_RK_b})
\begin{equation}
b_3=1-c_2-b_4.\label{eq_43_c2c3_b3}
\end{equation}
From Eq.~(\ref{eq_oc_43_f1_2}) (consistent also with the general case (\ref{eq_a_bc_rec}))
\begin{equation}
a_{43}=c_4-c_2.\label{eq_43_c2c3_a43}
\end{equation}
Equations~(\ref{eq_oc_RK_bc}) and (\ref{eq_oc_RK_bc2}) are solved for $b_2$ and $b_4$:
\begin{eqnarray}
b_2 &=& \frac{c_4/2-1/3}{c_2(c_4-c_2)}-b_3,\label{eq_43_c2c3_b2}\\
b_4 &=& \frac{1/3-c_2/2}{c_4(c_4-c_2)}.\label{eq_43_c2c3_b4}
\end{eqnarray}
The system of equations (\ref{eq_oc_43_f1_3}) and (\ref{eq_oc_RK_bac}) is simultaneously solved for $a_{32}$ and $a_{42}$:
\begin{eqnarray}
a_{32} &=& \frac{b_4(b_2+b_3)a_{43}-b_3/(6c_2)}{b_4a_{43}-b_3(1-c_2)},\label{eq_43_c2c3_a32}\\
a_{42} &=& \frac{1}{b_3}[b_2a_{43}-(a_{43}-b_3)a_{32}].\label{eq_43_c2c3_a42}
\end{eqnarray}
After choosing $c_2$ and $c_4$ proceed in the order (\ref{eq_43_c2c3_b4}), (\ref{eq_43_c2c3_b3}), (\ref{eq_43_c2c3_b2}), (\ref{eq_43_c2c3_b1}), (\ref{eq_43_c2c3_a43}), (\ref{eq_43_c2c3_a32}) and (\ref{eq_43_c2c3_a42}).

\subsubsection{Case $c_3=c_4$}

$c_2\neq c_3$ are taken as free parameters.
$a_{32}$ and $a_{42}$ are the same as in the general case (\ref{eq_a_bc_rec}):
\begin{eqnarray}
a_{32}&=&\frac{b_2}{b_1+b_2-c_2}(c_3-c_2),\label{eq_43_c3c4_a32}\\
a_{42}&=&\frac{b_2}{b_1+b_2-c_2}(c_3-c_2-a_{43}).\label{eq_43_c3c4_a42}
\end{eqnarray}
Eq.~(\ref{eq_43_c3c4_a32}) is subtracted from (\ref{eq_43_c3c4_a42}) and substituted in Eq.~(\ref{eq_oc_43_f1_3}). Canceling $a_{43}\neq0$ gives $b_1+b_2+b_3=c_3$ from which (with Eq.~(\ref{eq_oc_RK_b}))
\begin{eqnarray}
b_4&=&1-c_3,\label{eq_43_c3c4_b4}\\
b_1&=&c_3-b_2-b_3.\label{eq_43_c3c4_b1}
\end{eqnarray}
Solving Eqs.~(\ref{eq_oc_RK_bc}), (\ref{eq_oc_RK_bc2}) for $b_2$ and $b_3$:
\begin{eqnarray}
b_2&=&\frac{c_3/2-1/3}{c_2(c_3-c_2)},\label{eq_43_c3c4_b2}\\
b_3&=&\frac{1/3-c_2/2}{c_3(c_3-c_2)}-b_4.\label{eq_43_c3c4_b3}
\end{eqnarray}
Solving Eq.~(\ref{eq_oc_RK_bac}) for $a_{43}$:
\begin{equation}
a_{43}=\frac{(1/6-b_3a_{32}c_2)(b_1+b_2-c_2)-b_4b_2c_2(c_3-c_2)}{b_4[(b_1+b_2-c_2)c_3-b_2c_2]}.\label{eq_43_c3c4_a43}
\end{equation}
After choosing $c_2$ and $c_3$ proceed in the order (\ref{eq_43_c3c4_b2}), (\ref{eq_43_c3c4_b4}), (\ref{eq_43_c3c4_b3}), (\ref{eq_43_c3c4_b1}), (\ref{eq_43_c3c4_a32}), (\ref{eq_43_c3c4_a43}) and (\ref{eq_43_c3c4_a42}).

\section{Improved precision for the coefficients of the (6,4) scheme of Ref.~\cite{BERLAND20061459}}
\label{sec_app_bbb}
The coefficients below were computed by solving the 2N-storage order conditions and two additional constraints used in Ref.~\cite{BERLAND20061459} with a standard Newton-Raphson method with backtracking line search~\cite{NumRecBook}, coded with the arbitrary precision arithmetic library MPFR~\cite{MPFR} and using the coefficients of Ref.~\cite{BERLAND20061459} as the initial guess. The value of $B_6$ is fixed to $0.27$, as was done in Ref.~\cite{BERLAND20061459}. The Newton-Raphson method converges to the accuracy of satisfying the order conditions at $10^{-300}$ in just four iterations.

\begin{eqnarray}
c_2 &=& \phantom{-}3.291860514560574016139360757085052620500596e-02,\nonumber\\
c_3 &=& \phantom{-}2.493517233431018504774294242339061755717526e-01,\nonumber\\
c_4 &=& \phantom{-}4.669117050548576634478026408182787823664122e-01,\nonumber\\
c_5 &=& \phantom{-}5.820304140439261598301282787623770385763741e-01,\nonumber\\
c_6 &=& \phantom{-}8.472529837826966533345857631306276101828820e-01,\nonumber\\
A_1 &=& \phantom{-}0,\nonumber\\
A_2 &=& -7.371013927959100015085736294563710861301655e-01,\nonumber\\
A_3 &=& -1.634740794340906961222612899974121227203739e+00,\nonumber\\
A_4 &=& -7.447390037800703313971792823734483498376512e-01,\nonumber\\
A_5 &=& -1.469897351521944371244484234187043583134644e+00,\nonumber\\
A_6 &=& -2.813971388035238894872690695659944758090490e+00,\nonumber\\
B_1 &=& \phantom{-}3.291860514560574016139360757085052620500596e-02,\nonumber\\
B_2 &=& \phantom{-}8.232569981988439778822317874254015260794315e-01,\nonumber\\
B_3 &=& \phantom{-}3.815309489002858170631520216481864120871775e-01,\nonumber\\
B_4 &=& \phantom{-}2.000922131840258454393248810001898523823106e-01,\nonumber\\
B_5 &=& \phantom{-}1.718581042714403494253985915871400632540402e+00,\nonumber\\
B_6 &=& \phantom{-}2.700000000000000000000000000000000000000000e-01,\nonumber
\end{eqnarray}

\section{Matlab script for checking the special cases $b_i=0$ in converting between the standard and 2N-storage coefficients}
\label{sec_app_matlab}
The following Matlab script illustrates the conversion and use of proper and improper formulas, discussed in Sec.~\ref{sec_check_conv}.
\begin{verbatim}
% check coefficient conversion both ways:
% (a,b) -> (A,B) -> (a,b)
function main

  for irpt=1:3
    if irpt == 1
      fprintf( "\nLow-storage Runge-Kutta (4,3) b3=0 case\n\n" );
      [ s, a, b, c ] = setup_lsrk_43_b3_0();
    elseif irpt == 2
      fprintf( "\nLow-storage Runge-Kutta (5,3) b4=0 case\n\n" );
      [ s, a, b, c ] = setup_lsrk_53_b4_0();
    else
      fprintf( "\nLow-storage Runge-Kutta (5,3) b3=0 case\n\n" );
      [ s, a, b, c ] = setup_lsrk_53_b3_0();
    end

    fprintf( "Original Butcher tableau\n" );
    print_butcher( s, a, b, c );
    fprintf( "\n" ); 
    [ alpha, beta ] = convert_a_to_alpha( s, a, b );
    check_oc_standard_3( s, a, b, c );
    fprintf( "\n" );
    check_oc_lowstorage( s, alpha, beta );
    fprintf( "\n" );
    [ A, B ] = convert_a_to_A( s, a, b, c );
    [ AW, BW ] = convert_a_to_A_Williamson( s, a, b, c );
    fprintf( "Correct A, B:\n" );
    for i=1:s
      fprintf( "A%d = %f  B%d = %f\n", i, A(i), i, B(i) );
    end
    fprintf( "\n" );
    fprintf( "Williamson's A, B:\n" );
    for i=1:s
      fprintf( "A%d = %f  B%d = %f\n", i, AW(i), i, BW(i) );
    end
    fprintf( "\n" );
    [ a2, b2, c2 ] = convert_A_to_a( s, A, B );
    fprintf( "Converted Butcher tableau\n" );
    print_butcher( s, a2, b2, c2 );
    fprintf( "\n" );
    [ alpha2, beta2 ] = convert_a_to_alpha( s, a2, b2 );
    check_oc_standard_3( s, a2, b2, c2 );
    fprintf( "\n" );
    check_oc_lowstorage( s, alpha2, beta2 );
    fprintf( "\n" );
    [ a2, b2, c2 ] = convert_A_to_a( s, AW, BW );
    fprintf( "Williamson converted Butcher tableau\n" );
    print_butcher( s, a2, b2, c2 );
    fprintf( "\n" );
    [ alpha2, beta2 ] = convert_a_to_alpha( s, a2, b2 );
    check_oc_standard_3( s, a2, b2, c2 );
    fprintf( "\n" );
    check_oc_lowstorage( s, alpha2,beta2 );
    fprintf( "\n" );
  end
end
function [ s, a, b, c ] = setup_lsrk_43_b3_0()
  s = 4;
  a = zeros(s,s);
  b = zeros(s,1);
  c = zeros(s,1);
  a(2,1) = 1/2; a(3,1) = 2/9; a(3,2) = 1/3;
  a(4,1) = 3/176; a(4,2) = 51/88; a(4,3) = 27/176;
  b(1) = 2/9; b(2) = 1/3; b(3) = 0; b(4) = 4/9;
  for i=2:s
    for j=1:i-1
      c(i) = c(i) + a(i,j);
    end
  end
end
function [ s, a, b, c ] = setup_lsrk_53_b4_0()
  s = 5;
  a = zeros(s,s);
  b = zeros(s,1);
  c = zeros(s,1);
  a(2,1) = 1/3; a(3,1) = 1/8; a(3,2) = 3/8;
  a(4,1) = 1/18; a(4,2) = 1/2; a(4,3) = 2/9;
  a(5,1) = 81/328; a(5,2) = 51/328; a(5,3) = -16/41; a(5,4) = 81/82;
  b(1) = 1/18; b(2) = 1/2; b(3) = 2/9; b(4) = 0; b(5) = 2/9;
  for i=2:s
    for j=1:i-1
      c(i) = c(i) + a(i,j);
    end
  end
end
function [ s, a, b, c ] = setup_lsrk_53_b3_0()
  s = 5;
  a = zeros(s,s);
  b = zeros(s,1);
  c = zeros(s,1);
  a(2,1) = 1/6; a(3,1) = 2/15; a(3,2) = 1/5;
  a(4,1) = 13/60; a(4,2) = -3/10; a(4,3) = 3/4;
  a(5,1) = 9/80; a(5,2) = 13/40; a(5,3) = -3/16; a(5,4) = 1/2;
  b(1) = 2/15; b(2) = 1/5; b(3) = 0; b(4) = 2/5; b(5) = 4/15;
  for i=2:s
    for j=1:i-1
      c(i) = c(i) + a(i,j);
    end
  end
end
function [ alpha, beta ] = convert_a_to_alpha( s, a, b )
  alpha = zeros(s,s);
  beta = zeros(s,1);
  for i=2:s
     for j=1:i-1
       alpha(i,j) = a(i,j) - a(i-1,j);
     end
  end
  for i=1:s
    beta(i) = b(i) - a(s,i);
  end
end
function [ A, B ] = convert_a_to_A( s, a, b, c )
  A = zeros(s,1);
  B = zeros(s,1);
  A(1) = 0;
  for i=2:s
    A(i) = ( b(i-1) - a(s,i-1) ) / ( b(i) - a(s,i) );
  end
  for i=1:s-1
      B(i) = a(i+1,i);
  end
  B(s) = b(s);
end
function [ A, B ] = convert_a_to_A_Williamson( s, a, b, c )
  A = zeros(s,1);
  B = zeros(s,1);
  for i=1:s-1
    B(i) = a(i+1,i);
  end
  B(s) = b(s);
  A(1) = 0;
  for i=2:s
    if b(i) == 0
      A(i) = ( a(i+1,i-1) - c(i) ) / B(i);
    else
      A(i) = ( b(i-1) - B(i-1) ) / b(i);
    end
  end
end
function [ a, b, c ] = convert_A_to_a( s, A, B )
  a = zeros(s,s);
  b = zeros(s,1);
  c = zeros(s,1);
  for i=2:s
    a(i,i-1) = B(i-1);
    c(i) = a(i,i-1);
    for j=i-2:-1:1
      a(i,j) = A(j+1)*a(i,j+1) + B(j);
      c(i) = c(i) + a(i,j);
    end
  end
  b(s) = B(s);
  for i=s-1:-1:1
    b(i) = A(i+1)*b(i+1) + B(i);
  end
end
function check_oc_standard_3( s, a, b, c )
  fprintf( "Standard order conditions:\n" );
  sum = 0;
  for i=1:s
    sum = sum + b(i);
  end
  fprintf( "b_i          = %f\n", sum );
  sum = 0;
  for i=2:s
    sum = sum + b(i)*c(i);
  end
  fprintf( "b_i*c_i      = %f\n", sum );
  sum = 0;
  for i=2:s
    sum = sum + b(i)*c(i)^2;
  end
  fprintf( "b_i*c_i^2    = %f\n", sum );
  sum = 0;
  for i=3:s
    for j=2:i-1
      sum = sum + b(i)*a(i,j)*c(j);
    end
  end
  fprintf( "b_i*a_ij*c_j = %f\n", sum );
end
function check_oc_lowstorage( s, alpha, beta )
  fprintf( "2N-storage order conditions:\n" );
  for j=1:s-2
    for i=j+2:s
      sum = beta(j+1)*alpha(i,j) - beta(j)*alpha(i,j+1);
      fprintf( "beta%d*alpha%d%d - beta%d*alpha%d%d = %e\n", ...
        j+1, i, j, j, i, j+1, abs(sum) );
    end
  end
end
function print_butcher( s, a, b, c )
  for i=2:s
    fprintf( "%.6f  |", c(i) );
    for j=1:i-1
        fprintf( "  %.6f", a(i,j) );
    end
    fprintf( "\n" );
  end
  fprintf( "----------|" );
  for i=1:s
    fprintf( "----------" );
  end
  fprintf( "\n" );
  fprintf( "          |" );
  for i=1:s
      fprintf( "  %.6f", b(i) );
  end
  fprintf( "\n" );
end
\end{verbatim}

\bibliography{lie_int}

\end{document}